\documentclass[11pt]{article}
\usepackage[utf8]{inputenc}
\usepackage{lmodern}
\usepackage{subfiles}
\usepackage{enumitem}
	\setenumerate{label={\normalfont(\roman*)}, itemsep=0em} 
       \usepackage{pgf,tikz}
\usetikzlibrary{arrows} 
\usepackage{amsfonts}
\usepackage{amsthm}
\usepackage{amsmath}
\usepackage{amssymb}
\usepackage{amscd}
\usepackage{mathrsfs}
\usepackage{mathtools}
\usepackage{bm}
\usepackage{esint}

\usepackage[margin=3cm]{geometry}
\usepackage{indentfirst}
\usepackage{graphicx}
\usepackage{graphics}
\usepackage{lscape}
\usepackage{tikz-cd}
\usepackage{color}
\usepackage{pict2e}
\usepackage{epic}
\usepackage{epstopdf}
\usepackage{titlesec, titlefoot}
	\titleformat{\section}[block]{\Large\bfseries\filcenter}{\thesection}{1em}{}
\usepackage{commath}
\usepackage{float}
\usepackage{caption}
\usepackage{etoolbox}
\usepackage[affil-it]{authblk}
\usepackage{combelow}

\usepackage[hidelinks]{hyperref}
\hypersetup{bookmarksopen=true} 
\usepackage{hypcap}

\graphicspath{{./Pictures/}}
\allowdisplaybreaks

\usepackage{listings}
\usepackage{fancybox}

\theoremstyle{plain}

\renewcommand*\thesection{\arabic{section}}
\numberwithin{equation}{section} 

\theoremstyle{plain}
\newtheorem{thm}{Theorem}

\newtheorem{lemma}[thm]{Lemma}

\newtheorem{corollary}[thm]{Corollary}
\newtheorem{theorem}[thm]{Theorem}

\numberwithin{thm}{section}

\expandafter\let\expandafter\oldproof\csname\string\proof\endcsname
\let\oldendproof\endproof
\renewenvironment{proof}[1][\proofname]{%
  \oldproof[\upshape \bfseries #1]%
}{\oldendproof}

\makeatletter
\def\@makechapterhead#1{%
  \vspace*{50\p@}%
  {\parindent \z@ \raggedright \normalfont
    \interlinepenalty\@M
    \Huge\bfseries  \thechapter.\quad #1\par\nobreak
    \vskip 40\p@
  }}
\makeatother

\newcommand{\reqnomode}{\tagsleft@false}
\def\dy{\,{\rm d}y}
\def \d{\,{\rm d}}

\allowdisplaybreaks
\makeatletter
\DeclareRobustCommand*{\bfseries}{%
  \not@math@alphabet\bfseries\mathbf
  \fontseries\bfdefault\selectfont
  \boldmath
}

\makeatother

\newlength{\defbaselineskip}
\newcommand{\N}{\mathbb{N}}
\newcommand\eps\varepsilon
\def\mean#1{\mathchoice%
          {\mathop{\kern 0.2em\vrule width 0.6em height 0.69678ex depth -0.58065ex
                  \kern -0.8em \intop}\nolimits_{\kern -0.4em#1}}%
          {\mathop{\kern 0.1em\vrule width 0.5em height 0.69678ex depth -0.60387ex
                  \kern -0.6em \intop}\nolimits_{#1}}%
          {\mathop{\kern 0.1em\vrule width 0.5em height 0.69678ex
              depth -0.60387ex
                  \kern -0.6em \intop}\nolimits_{#1}}%
          {\mathop{\kern 0.1em\vrule width 0.5em height 0.69678ex depth -0.60387ex
                  \kern -0.6em \intop}\nolimits_{#1}}}

\def\eqn#1$$#2$${\begin{equation}\label#1#2\end{equation}}
\newcommand\R{\mathbb{R}}

\newcommand{\Div}{\mathrm{div}\,}

\def \tp{\textup}
\def \p{\partial}
\def \e{\varepsilon}
\def \D{\mathrm{D}}

\def \LL{\mathrm L}

\newcommand\restr[2]{{
  \left.\kern-\nulldelimiterspace 
  #1 
  \vphantom{|} 
  \right|_{#2} 
  }}

\title{Geometric linearisation for optimal transport with strongly $p$-convex cost}
\author{Lukas Koch}

\affil[1]{\small Max Planck Institute for Mathematics in the Sciences, 04103 Leipzig,Germany
\protect \\
  {\tt{lkoch@mis.mpg.de}}
  \vspace{1em} \ }

\usepackage{etoolbox}
\makeatletter
\patchcmd{\@adjustvertspacing}
  {\jot\baselineskip \divide\jot 4}
  {\jot=.5\baselineskip}
  {}{}
\makeatother

\makeatother
\begin{document}

\maketitle
\begin{abstract}
We prove a geometric linearisation result for minimisers of optimal transport problems where the cost-function is strongly $p$-convex and of $p$-growth. Initial and target measures are allowed to be rough, but are assumed to be close to Lebesgue measure.
\end{abstract}

\section{Introduction}\label{sec:intro}
The study of the optimal transport problem:
\begin{align}\label{eq:problem}
\min_{\pi\in \Pi(\lambda,\mu)}\int_{\R^d\times \R^d} c(x-y)\d\pi
\end{align}
is well established. Here $\lambda,\mu$ are two (finite) non-negative measures on $\R^d$ satisfying $\lambda(\R^d)=\mu(\R^d)$. We refer the reader to \cite{Villani2009} and \cite{Santambrogio2015} for an introduction and overview of the literature. When solutions take the form of a transport map $\pi = (\tp{Id}\times T)_\#\mu$, under mild assumptions, minimisers are characterised by satisfying the Euler-Lagrange equation
\begin{align}\label{eq:EL}
\mu(T(x))\tp{det}\, (\D T(x)) = \lambda(x)
\end{align}
as well as the additional structure condition
\begin{align}
T(x)=x+ \nabla c^\ast(\D\phi),
\end{align}
where $\phi$ is a $c$-concave function and $c^\ast$ denotes the convex conjugate of $c$. Assuming $\mu\sim\lambda\sim 1$ and linearising the geometric nonlinearity in \eqref{eq:EL}, that is formally expanding $\tp{det}(\tp{Id}+A)=1+\tp{tr}\, A+\ldots$, we find that
\begin{align}\label{eq:linearised}
\tp{div}\, \nabla c^\ast(\D\phi)=\lambda-\mu.
\end{align}
Thus, at least formally, we expect solutions of \eqref{eq:problem} to be well approximated by solutions of \eqref{eq:linearised}. Note that in general \eqref{eq:linearised} is a nonlinear equation. Thus we refer to the process of moving from \eqref{eq:problem} to \eqref{eq:linearised} as geometric linearisation. The aim of this paper is to make this connection rigorous.
We show the following:
\begin{theorem}\label{thm:main}Let $1<p<\infty$.
Suppose $c\colon \R^d\to \R$ is a strongly $p$-convex cost function of controlled-duality $p$-growth, that is satisfying \eqref{ass:elliptic}-\eqref{eq:controlledGrowth}. Let $\pi$ be a minimiser of \eqref{eq:problem} for some non-negative measures $\lambda$, $\mu$ satisfying $\lambda(\R^d)=\mu(\R^d)$. Denote
\begin{gather*}
E(R) := \frac 1 {\lvert B_{R}R^p\rvert}\int_{(B_{R}\times \R^d)\cup (\R^d\times B_{R})}c(x-y)\d\pi \\
D(R) := \frac 1 {\lvert B_{R}\rvert} W_{p}^p(\lambda\llcorner B_R,\kappa_{\lambda,R}\d x\llcorner B_R)+\frac{R^p}{\kappa_{\lambda,R}^{p-1}}(\kappa_{\lambda,R}-1)^p\\
\qquad\qquad+\frac 1 {\lvert B_{R}\rvert} W_{p}^p(\mu\llcorner B_R,\kappa_{\mu,R}\d x \llcorner B_R)+\frac{R^p}{\kappa_{\mu,R}^{p-1}}(\kappa_{\mu,R}-1)^p.
\end{gather*}
Here $\kappa_{\lambda,R} = \frac{\lambda(B_R)}{\lvert B_R\rvert}$ and $\kappa_{\mu,R} = \frac{\mu(B_R)}{\lvert B_R\rvert}$.
Then, for every $\tau>0$, there exists $\e(\tau)>0$ such that if ${E(4)+D(4)\leq \e}$\footnote{For applications it may be useful to note that the scale $4$ is not important and it is straightforward to carry out the proofs under the assumption that $E(r)+D(r)\leq \e$ for some $r>3$. In that case $\e$ will depend on $r-3$.}, then there exists a radius $R\in(2,3)$, $c\in \R$ and $\phi$ satisfying
\begin{align*}
-\Div \nabla c^\ast(\D\phi) = c \text{ in } B_R 
\end{align*}
such that
\begin{align}\label{eq:mainInequality}
\int_{(B_1\times \R^d)\cup (\R^d\times B_1)}c(x-y-\nabla c^\ast(\D\phi))\d\pi\lesssim \tau E(4)+ D(4).
\end{align}
Moreover
\begin{align*}
\sup_{B_1}|\D\phi|^{p'}+\int_{B_R}\lvert\D\phi\rvert^{p'}\d x\lesssim E(4)+D(4).
\end{align*}
\end{theorem}

We remark that we explain our assumptions on the cost function in detail in Section \ref{sec:costFunction}. Theorem \ref{thm:main} states that if at some scale the local transportation cost $E(R)$ (and the data term) are small, then at a smaller scale, the transportation plan is well-approximated by $\nabla c^\ast(\D\phi)$, in the sense that estimate \eqref{eq:mainInequality} holds. Note in particular that as a consequence of \eqref{eq:mainInequality} and \eqref{ass:Cgrowth}, also
\begin{align*}
\int_{(B_1\times \R^d)\cup (\R^d\times B_1)}\lvert x-y-\nabla c^\ast(\D\phi)\rvert^p d\pi\lesssim \tau E(4)+D(4).
\end{align*}
Thus Theorem \ref{thm:main} makes the intuition leading to \eqref{eq:linearised} rigorous.

Traditionally, \eqref{eq:problem} has been approached via the study of \eqref{eq:EL} using the theory of fully nonlinear elliptic equations developed by \textsc{Caffarelli}, see e.g. \cite{Figalli2010,DePhilippis2015} and the references therein. Recently, an alternative approach using variational techniques has been developed by \textsc{Goldman and Otto} in \cite{Goldman2018}. There, partial $C^{1,\alpha}$-regularity for solutions to \eqref{eq:problem} in the case of H\"older-regular densities $\lambda$, $\mu$ and quadratic cost function $c(x-y)=\frac 1 2 |x-y|^2$ was proven. The key tool in the proof was a version of Theorem \ref{thm:main} in this setting. In later papers, continuous densities \cite{Goldman2020a}, rougher measures \cite{Goldman2021}, more general cost functions (albeit still close to the quadratic cost functional) \cite{Otto2021}, as well as almost-minimisers of the quadratic cost functional \cite{Otto2021} were considered. The quadratic version of Theorem \ref{thm:main} was also used to provide a more refined linearisation result of \eqref{eq:EL} in the quadratic set-up in \cite{Goldman2021} and of a similar statement in the context of optimal matching in \cite{Goldman2022}. Finally, quadratic versions of Theorem \ref{thm:main} played a key role in disproving the existence of a stationary cyclically monotone Poisson matching in $2$-d \cite{Huesmann2021}. 

We remark that very little information is available about the regularity of minimiser of \eqref{eq:problem} already in the simplest degenerate/singular case $c(x-y)=\frac{|x-y|^p}{p}$. In order to attempt to extend the techniques of \cite{Goldman2018} to this setting, an essential first step is Theorem \ref{thm:main}. This result will also play a key role in extending the results of \cite{Huesmann2021} to $p$-costs with $p\neq 2$ \cite{Huesmann2023}.

The strategy of proof is similar to that used in \cite{Goldman2021}, although with a number of simplifications. Further, we point the reader to \cite{Koch2023} where a detailed account of the proof of Theorem \ref{thm:main} and the motivations behind the strategy are given in the quadratic case. The heart of the proof is contained in Section \ref{sec:proof}. The key insight Lemma \ref{lem:orthog} is a consequence of the strong $p$-convexity, which allows to estimate (up to error terms) the left-hand side of the main estimate in Theorem \ref{thm:main} by a sum of two terms: on the one hand the difference between the local transportation energy of $\pi$ and a dual form of the energy of $\phi$:
\begin{align}\label{eq:diffEnergies}
\int_{(B_R\times \R^d)\cup (\R^d\times B_R)} c(x-y)\d\pi-\int_{B_R} c(\nabla c^\ast(\D\phi))\d x,
\end{align}
and on the other hand
\begin{align}\label{eq:PDEterm}
\int_{\{\exists t\colon X(t)\in B_R\}}\int_{\tau}^{\sigma}\langle y-x-\nabla c^\ast(\D\phi(X(t)),\D\phi(X(t))\rangle\d t \d\pi.
\end{align}
Here we write $X(t)=(1-t)x+ty$ and set ${\tau = \inf\{t\in[0,1]\colon X(t)\in B_R\}}$, as well as $\sigma = \sup\{t\in[0,1]\colon X(t)\in B_R\}$.
Lemma \ref{lem:orthog} replaces the quasi-orthogonality property, which was employed in the quadratic case. The quasi-orthogonality property relied on expanding the squares, which is not an available tool if $p\neq 2$. 

From formal calculations, which we make rigorous in Lemma \ref{lem:secondTerm}, \eqref{eq:PDEterm} will be small, if $\phi$ solves the Neumann problem
\begin{align}\label{eq:formal}
-\Div \nabla c^\ast(\D\phi)=\mu-\lambda\quad& \text{ in } B_R\\
\nabla c^\ast(\D\phi)\cdot\nu = g_R-f_R &\text{ on } \p B_R,
\end{align}
where $f_R,g_R$ are functions tracking the location of $X(\tau)$ and $X(\sigma)$, respectively. We formally define $f_R,g_R$ in \eqref{eq:2}. However, as written, the problem is not well-posed and solutions do not possess sufficient regularity to carry out the necessary estimates. Hence, we will actually work with approximations of $f_R$ and $g_R$, which we construct in Section \ref{sec:boundaryData}. Controlling the error made in this approximation, will require us to enlarge the domain on which we work and to choose a suitable radius $R\in[2,3]$. This explains the presence of $R$ in the formulation of Theorem \ref{thm:main}. The estimate of \eqref{eq:PDEterm} is finally carried out in Lemma \ref{lem:secondTerm} in Section \ref{sec:proof}.

The idea in estimating \eqref{eq:diffEnergies} is to relate the first term with the value of a localised optimal transportation problem. Then an appropriate competitor can be constructed using $\phi$ in order to estimate \eqref{eq:diffEnergies}. We carry out the first step in Section \ref{sec:localisation}, while the second is obtained in Lemma \ref{lem:firstTerm} in Section \ref{sec:proof}.

In order to carry out the estimates, both of \eqref{eq:diffEnergies} and \eqref{eq:PDEterm}, two ingredients are essential. We require elliptic regularity estimates which follow from strict $p^\prime$-convexity of $c^\ast$. We explain how to obtain these in Section \ref{sec:regularity}. In the quadratic case, the relevant equation is Laplace equation. Hence solutions are harmonic and hence very regular- the proof in \cite{Goldman2021} requires $C^3$-regularity of solutions! Already in the case $c(x-y)=\frac{|x-y|^p}{p}$ with $p\neq 2$, the best regularity that is known for solutions to \eqref{eq:linearised} in general is $C^{1,\beta}$-regularity for some $\beta>0$ \cite{Uraltseva1968,Lewis1983,Tolksdorf1983}. Thus, at various places in the proof, more careful estimates are needed. 

The second ingredient is to obtain a $L^\infty$-bound for minimisers of \eqref{eq:problem} in the small-energy regime, see Section \ref{sec:Linfty}. In the quadratic case, this relies on the monotonicity (in the classical sense) property of solutions. In the non-quadratic case, $c$-monotonicity needs to be used directly. Focusing on $p$-homogeneous convex cost functions with $p>1$, $L^\infty$-bounds were obtained in \cite{Gutierrez2021}. Note that \cite{Gutierrez2021} obtained $L^\infty$-bounds in all energy regimes whereas the $L^\infty$-bounds obtained in this paper only cover the small-energy regime. A further difference is that in this paper, we obtain the $L^\infty$-bound as a consequence of the strong $p$-convexity of the energy, whereas \cite{Gutierrez2021} relies on the homogeneity of the cost. Nevertheless, in the small-energy regime and for cost functions covered by the assumptions both of this paper and of \cite{Gutierrez2021}, the obtained bounds agree. In particular, this is the case for the important example of $p$-cost $c(x-y)=\frac 1 p\lvert x-y\rvert^p$ with $p>1$. 
 
Finally, we comment why we restrict our attention to cost functions of the form $c(x-y)$. This is due to the fact that our proof relies on the availability of a dynamical formulation. As \eqref{eq:PDEterm} hints, we want to identify points $(x,y)\in \mathrm{spt}\,\pi$ with the trajectory ${X(t)=tx+(1-t)y}$. This is related to the Benamou-Brenier formulation of optimal transport, c.f. \cite{Gangbo1995}, which in our case states that \eqref{eq:problem} can be alternatively characterised as
\begin{align}\label{eq:eulerianFormulation}
\min_{(j,\rho)}\left\{\int c\left(\frac{\d j_t}{\d \rho_t}\right)\d\rho\colon \partial_t \rho_t+\Div j_t = 0,\, \rho_0=\lambda,\, \rho_1=\mu\right\}
\end{align}
Here $\frac{\d j_t}{\d \rho_t}$ denotes the Radon-Nikodym derivative. We refer the reader to \eqref{eq:BenamouBrenierDefinition} for an explanation on how to make sense of \eqref{eq:eulerianFormulation} rigorously. This alternative, dynamical formulation of optimal transport is only available for costs of the form $c(x-y)$ where $c$ is convex.

The outline of the paper is as follows: In the remainder of this section, we make precise our assumptions on the cost function (Section \ref{sec:costFunction}) and collect the elliptic regularity statements we require (Section \ref{sec:regularity}). After collecting notation and some elementary results on optimal transportation in Section \ref{sec:prelim}, we collect the statements we require in order to prove Theorem \ref{thm:main} in Section \ref{sec:proof}: in Section \ref{sec:Linfty}, we prove the $L^\infty$-estimate. In Section \ref{sec:localisation}, we prove a localisation result on optimal transportation costs. Next, we approximate the boundary data of \eqref{eq:formal} in Section \ref{sec:boundaryData}. For technical reasons, we also need to localise the data-term $D$, which we do in Section \ref{sec:data}.

\subsection{Assumptions on the cost function and its dual}\label{sec:costFunction}
In this section, we explain our assumptions on the cost function $c$. In order to keep the statements of our results short, the conditions \eqref{ass:elliptic}--\eqref{eq:controlledGrowth} below will be assumed to hold throughout the entire paper. We consider cost-functions modeled on the $p$-cost, $c(x)=\lvert x\rvert^p$ with $p>1$. This is also the primary example of cost functions we have in mind. We emphasize however that the cost functions we consider need not be homogeneous. In fact the assumptions we impose are standard within elliptic regularity theory, see e.g. \cite{Giusti2003,Mingione2006}. Let $p\in(1,\infty)$. We consider a $C^1$-cost function $c\colon \R^d\to \R$ satisfying the following properties:
There is $\Lambda\geq 1$ such that
\begin{enumerate} 
\item $c$ is strongly $p$-convex: for any $x,y\in \R^d$ and $\tau\in[0,1]$,
\begin{align}\label{ass:elliptic}
\Lambda^{-1} \tau(1-\tau) V_p(x,y)+c(\tau x+(1-\tau)y)\leq \tau c(x)+(1-\tau)c(y).
\end{align}
where 
$$
V_p(x,y)=(|x|^2+|y|^2)^\frac{p-2} 2|x-y|^2.
$$
\item $c$ has $p$-growth: for any $x\in \R^d$, 
\begin{align}\label{ass:growth}
\Lambda^{-1} |x|^p\leq c(x)\leq \Lambda |x|^p.
\end{align}
\item for any $x,y\in \R^d$,
\begin{align}\label{ass:Cgrowth}
|c(x)-c(y)|\leq \Lambda U_p(x,y)
\end{align}
where
$$
U_p(x,y)= (|x|+|y|)^{p-1}|x-y|.
$$
\item $c$ satisfies controlled $p$-growth: For any $x,y\in \R^d$,
\begin{align}\label{eq:controlledGrowth}
|\nabla c(x)-\nabla c(y)|\leq \Lambda (\lvert x\rvert+\lvert y\rvert)^{p-2}\lvert x-y\rvert.
\end{align}
\end{enumerate}

If the choice of $p$ is clear from the context, we will write $V=V_p$ and $U=U_p$. We further note the following inequality, valid for any $z_1,z_2,z_3\in \R^d$ and with implicit constants depending only on $p$ and $d$,
\begin{align}\label{eq:VDiff}
\lvert V_p(z_1,z_2)-V_p(z_1,z_3)\rvert \lesssim (\lvert z_1\rvert + \lvert z_2\rvert+\lvert z_3\rvert)^{p-1}\lvert z_2-z_3\rvert.
\end{align}
\eqref{eq:VDiff} follows from writing
\begin{align*}
V_p(z_1,z_2)-V_p(z_1,z_3) = \int_0^1 \langle \nabla_2 V_p(z_1,s z_2+(1-s)z_3), z_2-z_3\rangle \d t.
\end{align*}
Here $\nabla_2$ denotes a derivative with respect to the second variable of $V_p$. From elementary calculations
\begin{align*}
\lvert \nabla_2 V_p(z_1+s z_2+(1-s)z_3)\rvert \lesssim (\lvert z_1\rvert+\lvert z_2\rvert+\lvert z_3\rvert)^{p-1},
\end{align*}
which gives \eqref{eq:VDiff}.

\eqref{ass:elliptic} and the fact that $c$ is $C^1$ imply that for any $x,y\in \R^d$,
\begin{align}\label{ass:smoothness}
c(x)\geq c(y)+\langle \nabla c(y),x-y\rangle +\lambda V(x,y),\\
\langle \nabla c(x)-\nabla c(y),x-y\rangle \geq \lambda V(x,y).
\end{align}
This can be seen by arguing as in the $2$-convex, $2$-growth case, which can be found for example in \cite[Chapter IV.4.1]{HiriartUrruty1993b}.

\noindent We require some information on the convexity properties of the convex conjugate that follow by adaptions of the $2$-convex, $2$-growth theory in \cite{Rockafellar1970} and \cite{HiriartUrruty1993a}. For the convenience of the reader, we provide proofs of the statements we require. Introduce the convex conjugate $c^\ast$ defined on $\R^d$ via
$$
c^\ast(\xi)= \sup_{x} \langle \xi,x\rangle - c(x).
$$
We remark that since $c$ is strongly convex, $C^1$ and superlinear at infinity, $c^\ast$ is strongly convex, $C^1$ and superlinear at infinity. Moreover $\nabla c$ and $\nabla c^\ast$ are homeomorphisms of $\R^d$ and $\nabla c^\ast = (\nabla c)^{-1}$ \cite[Theorem 26.5]{Rockafellar1970}.
Note that due to \eqref{ass:growth}, we have for any $\xi\in \R^d$,
\begin{align*}
c^\ast(\xi)\leq \sup_x \langle \xi,x\rangle-\Lambda^{-1} \lvert x\rvert^p \lesssim \lvert \xi\rvert^{p^\prime}.
\end{align*}
A lower bound can be obtained similarly and we deduce:
\begin{align}\label{ass:growthDual}
|\xi|^{p'}\lesssim c^\ast(\xi)\lesssim |\xi|^{p'}.
\end{align}
Due to strict $p$-convexity of $c$, $c^\ast$ satisfies for any $\xi_1,\xi_2\in \R^d$,
\begin{align}\label{eq:c1growthDual}
|\nabla c^\ast(\xi_1)-\nabla c^\ast(\xi_2)|\lesssim(|\xi_1|+|\xi_2|)^{p'-2}|\xi_1-\xi_2|.
\end{align}
Indeed, for $\xi_1,\xi_2\in \R^d$, using \eqref{ass:smoothness} with the choice ${x=\nabla c^\ast(\xi_1)}$, $y=\nabla c^\ast(\xi_2)$ and Cauchy-Schwarz,
\begin{align*}
V_p(\nabla c^\ast(\xi_1),\nabla c^\ast(\xi_2))\lesssim& \langle \xi_1-\xi_2,\nabla c^\ast(\xi_1)-\nabla c^\ast(\xi_2)\rangle
\leq& \lvert \nabla c^\ast(\xi_1)-\nabla c^\ast(\xi_2)\rvert \lvert \xi_1-\xi_2\rvert.
\end{align*}
Re-arranging, we have
\begin{align}\label{eq:c1growthDual2}
\lvert \nabla c^\ast(\xi_1)-\nabla c^\ast(\xi_2)\rvert\lesssim \lvert \xi_1-\xi_2\rvert (\lvert \nabla c^\ast(\xi_1)\rvert^2+\lvert\nabla c^\ast(\xi_2)\rvert^2)^{\frac{2-p} 2}.
\end{align}
We claim that  for any $\xi\in \R^d$,
\begin{align}\label{eq:sizeDualGradient}
\lvert \xi\rvert^{p^\prime-1}\lesssim \lvert \nabla c^\ast(\xi)\rvert\lesssim \lvert \xi\rvert^{p^\prime-1}.
\end{align}
Then \eqref{eq:c1growthDual2} gives \eqref{eq:c1growthDual}. Since $\nabla c^\ast = (\nabla c)^{-1}$ and both maps are homeomorphisms, to show \eqref{eq:sizeDualGradient}, it suffices to show that for any $x\in \R^d$,
\begin{align}\label{eq:c1growthDual4}
\lvert x\rvert^{p-1}\lesssim \lvert \nabla c(x)\rvert\lesssim \lvert x\rvert^{p-1}.
\end{align}
Fix $x\in \R^d$. Since difference quotients of convex functions are non-decreasing, for any $h\in \R^d$, applying also \eqref{ass:Cgrowth},
\begin{align*}
\langle \nabla c(x),h\rangle \leq \frac{f(x+h)-f(x)}{\lvert h\rvert}\lesssim \frac{(\lvert x\rvert+\lvert h\rvert)^{p}}{\lvert h\rvert}.
\end{align*}
Applying the above with $h\to th$ and choosing $t$ such that $\lvert t h\rvert = \lvert x\rvert$, as $h$ was arbitrary, this gives the second inequality in \eqref{eq:c1growthDual4}. Note that in particular $\nabla c(0)=0$. Thus, using \eqref{ass:elliptic} with $y=0$ and Cauchy-Schwartz gives
\begin{align*}
\lvert x\rvert^p\lesssim \langle \nabla c(x),x\rangle\leq \lvert \nabla c(x)\rvert \lvert x\rvert.
\end{align*}
After rearranging, this proves the first inequality in \eqref{eq:c1growthDual4}. Hence \eqref{eq:c1growthDual} is established.

Finally, it follows from \eqref{eq:c1growthDual} and the intermediate value theorem that for $\xi_1,\xi_2\in \R^d$,
\begin{align}\label{eq:CgrowthDual}
\lvert c^\ast(\xi_1)-c^\ast(\xi_2)\rvert= \Big\lvert\int_0^1 \langle \nabla c^\ast(s\xi_1+(1-s)\xi_2),\xi_1-\xi_2\rangle\d s\Big\rvert \lesssim U_{p'}(\xi_1,\xi_2).
\end{align}

We also require that $c^\ast$ is $p^\prime$ convex, that is for some $C(p,\Lambda)>0$,
\begin{align}\label{eq:p'convex}
c \tau(1-\tau) V_{p'}(\xi_1,\xi_2)+c^\ast(\tau \xi_1+(1-\tau)\xi)2)\leq \tau c^\ast(\xi_1)+(1-\tau)c^\ast(\xi_2).
\end{align}
Indeed, we can use Taylor's theorem and \eqref{eq:controlledGrowth} to obtain for any $x,y\in \R^d$ and some $C=C(p,n)>0$,
\begin{align}\label{eq:p'convex2}
c(y)=& c(x)+\langle \nabla c(x),y-x\rangle + \int_0^1 \langle \nabla c(x+t(y-x))-\nabla c(x),y-x\rangle \mathrm{d}t\nonumber\\
\leq& c(x)+\langle \nabla c(x),y-x\rangle + C (\lvert x\rvert+\lvert x-y\rvert)^{p-2}\lvert y-x\rvert^2.
\end{align}
In order to estimate the integral, we used a well-known estimate, see e.g. \cite{Hamburger1992,Giaquinta1986}. Recall the Fenchel-Young inequality in the form
\begin{align}\label{eq:FenchelYoung}
c(\xi) +c^\ast(x)\geq \langle \xi,x\rangle \quad \forall \xi, x\in \R^d,
\end{align}
with equality if and only if $\xi = \nabla c^\ast(x)$. Hence, with the choice $x=\nabla c^\ast(\xi_1)$, \eqref{eq:p'convex2} gives
\begin{align}\label{eq:p'convex3}
-c^\ast(\xi_2)=& -\sup_y \{ \langle \xi_2,y\rangle-c(y)\}\nonumber\\
\leq& - c^\ast(\xi_1)-\sup_{y}\left\{ \langle \xi_2-\xi_1,y\rangle -C(\lvert \nabla c^\ast(\xi_1)\rvert+\lvert \nabla c^\ast(\xi_1)-y\rvert)^{p-2} \lvert y-\nabla c^\ast(\xi_1)\right\}.
\end{align}
Note that the supremum is nothing but $\left(C V_p(\nabla c^\ast(\xi_1),\cdot-\nabla c^\ast(\xi_1)\right)^\ast(\xi_2-\xi_1)$. Noting that for $x,y\in \R^d$,
\begin{align*}
(\lvert x\rvert+\lvert y\rvert)^{p-2}\lvert y\rvert^2 \lesssim \begin{cases}
					\lvert y\rvert^p \quad &\text{ if } \{\lvert x\rvert \leq \lvert y\rvert, p\geq 2\} \text{ or } \{\lvert x\rvert \geq \lvert y\rvert, p\leq 2\}\\
					\lvert x\rvert^{p-2} \lvert x\rvert^2 \quad & \text{ if } \{\lvert x\rvert \geq \lvert y\rvert, p\geq 2\} \text{ or } \{\lvert x\rvert \leq \lvert y\rvert, p\leq 2\}
				  \end{cases}
\end{align*}
and arguing case by case as for \eqref{ass:growthDual}, this shows
\begin{align*}
\left((\lvert x\rvert+\lvert \cdot\rvert)^{p-2}\lvert \cdot\rvert^2\right)^\ast(\xi)\gtrsim (\lvert x\rvert^{p-1}+\lvert \xi\rvert)^{p^\prime-2}\lvert \xi\rvert^2.
\end{align*}
In particular, recalling \eqref{eq:sizeDualGradient}, we deduce
\begin{align*}
V_p(\nabla c^\ast(\xi_1),\cdot)^\ast(\xi_2-\xi_1)\gtrsim (\lvert \nabla c^\ast(\xi_1)\rvert^{p-1}+\lvert \xi_1-\xi_2\rvert)^{p^\prime-2} \lvert \xi_1-\xi_2\rvert^2\gtrsim V_{p^\prime}(\xi_1-\xi_2).
\end{align*}
Employing the fact that for any convex $f\colon \R^d\to \R$ and any $x,\xi\in \R^d$, it holds that ${f(\cdot-x)^\ast(\xi) = f^\ast(\xi)+\langle x,\xi\rangle}$, we finally conclude,
\begin{align*}
\left(C V_p(\nabla c^\ast(\xi_1),\cdot-\nabla c^\ast(\xi_1)\right)^\ast(\xi_2-\xi_1)\gtrsim V_{p^\prime}(\xi_1-\xi_2)+\langle \nabla c^\ast(\xi_1),\xi_2-\xi_1).
\end{align*}
Combining this estimate with \eqref{eq:p'convex3} gives \eqref{eq:p'convex}.

\subsection{Regularity assumptions on the dual system}\label{sec:regularity}
Let $R\in (2,3)$. In this section, we state the regularity assumptions we make on distributional solutions $\phi\in W^{1,p'}(B)$ of the equation
\begin{align}\label{eq:dualEquation}
-\Div \nabla c^\ast(\D\phi) &= c_g \quad\text{ on } B_R\\
\nabla c^\ast(\D\phi)\cdot\nu &= g \quad\text{ on } \p B_R,
\end{align}
where $g\in L^{p}(B)$ and $c_g$ satisfies the compatibility condition $\lvert B_R\rvert c_g = \int_{\p B} g$. $\nu$ denotes the outward pointing normal vector on $\p B$. We will show that solutions exist and are unique up to a constant. Hence, we usually normalise solutions by requiring that $\int_B \phi = 0$. Fixing $g\in L^p(\p B)$, we denote by $\phi^r$ the solution satisfying $\int_B \phi^r=0$ of  \eqref{eq:dualEquation} with data $g^r$, where $g^r$ denotes convolution of $g$ with a smooth convolution kernel on $\partial B$ at scale $r$.

\begin{lemma}\label{lem:ellipticRegularity}
If $c^\ast$ satisfies \eqref{ass:elliptic}-\eqref{eq:controlledGrowth}, then solutions $\phi$ to \eqref{eq:dualEquation} exist, are unique up to constant and the following statements hold:
\begin{enumerate}
\item $\phi$ satisfies the following energy estimates:
\begin{align}
\label{eq:energy}\int_{B_R}|\D\phi|^{p'}\d x\lesssim \int_{\p B_R} |g|^p,\\
\label{eq:alternativeEnergy}\int_{B_R} c(\nabla c^\ast(\D\phi))\d x\lesssim \int_{\p B_R} |g|^p.
\end{align}
\item $\phi$ is Lipschitz regular in the interior of $B_R$: For any $r<R$,
\begin{align}\label{eq:interior}
\sup_{x\in B_r} \lvert \D\phi\rvert^{p^\prime} \lesssim_{R-r} \int_{\p B_R} \lvert g\rvert^p.
\end{align}
\item The difference between $\phi$ and $\phi^r$ is controlled: There exists $s=s(n,p)>0$ such that
\begin{align}\label{eq:diff}
\int_{B_R} |\D\phi-\D\phi^r|^{p'}\d x\lesssim r^s\int_{\p B_R} |g|^p.
\end{align}
\item $\D\phi^r$ is H\"older-regular up to the boundary: For any $\beta\in (0,1)$,
\begin{align}\label{eq:regularityPhiR}
r^\beta[\D\phi^r]_{C^{0,\beta}(B_R)}^{p'}+\sup_{B_R} |\D\phi^r|^{p'}\lesssim \frac 1 {r^{d-1}}\int_{\p B_R} |g|^{p}.
\end{align}
\item Let $\beta\in(0,1)$. Whenever $\phi\in C^{0,\beta}(B)$ for some ball $B$, 
\begin{align*}
[c^\ast(\D\phi)+c(\nabla c^\ast(\D\phi))]_{C^{0,\beta}(B)} \lesssim \|\D\phi\|_{L^\infty(B)}^{p'-1}[D\phi]_{C^{0,\beta}(B)}.
\end{align*}
\end{enumerate}
\end{lemma}
\begin{proof}
Note that in light of the results of Section \ref{sec:costFunction}, $c^\ast$ is $p^\prime$-convex and satisfies controlled $p^\prime$-growth. Hence, the statements we need to prove are largely standard.

The existence and uniqueness up to constant of solutions follows from the direct method. Testing the weak formulation of \eqref{eq:dualEquation} with $\phi$ and applying \eqref{eq:p'convex} in combination with H\"older's inequality, the trace estimate and Poincar\'{e} inequality (recall that $\int_B \phi = 0$) gives:
\begin{align*}
\|\D\phi\|_{L^{p^\prime}(B_R)}^{p^\prime}\lesssim& \int_{B_R} \langle \nabla c^\ast(\D\phi),\D\phi\rangle\d x = \int_{\p B_R} g \phi\leq \|\phi\|_{L^{p^\prime}(\p B_R)}\|g\|_{L^{p}(\p B_R)}\\
\lesssim&\|\phi\|_{W^{1,p^\prime}(B_R)}\|g\|_{L^{p}(\p B_R)}\lesssim \|\D\phi\|_{L^{p^\prime}(B_R)}\|g\|_{L^{p}(\p B_R)}.
\end{align*} 
Re-arranging this gives \eqref{eq:energy}. Using \eqref{ass:growth} and \eqref{eq:sizeDualGradient}, \eqref{eq:energy} implies that also
\begin{align*}
\int_{B_R} c(\nabla c^\ast(\D\phi))\d x\lesssim \int_{B_R} \lvert \nabla c^\ast(\D\phi)\rvert^p\d x\lesssim \int_{B_R}\lvert \D\phi\rvert^{p^\prime}\d x\lesssim \int_{\p B_R} |g|^p,
\end{align*}
that is \eqref{eq:alternativeEnergy} holds. 

\eqref{eq:interior} is proven in \cite{Lieberman1988} and \cite{Ladyzhenskaya1968}.  As the proofs are quite involved, we do not comment on them here. Instead we turn to \eqref{eq:diff}. We focus on the case $p^\prime \leq n$ as the other case is easier. Testing the equations for $\phi$ and $\phi^r$ with $\D\phi-\D\phi^r$ and applying \eqref{eq:p'convex} and H\"older's inequality, we find
\begin{align*}
\int_{B_R} V_{p^\prime}(\D\phi,\D\phi^r)\lesssim& \int_{B_R} \langle\nabla c^\ast(\D\phi)-\nabla c^\ast(\D\phi^r),\D\phi-\D\phi^r)\d x= \int_{\p B_R} (g-g^r)(\phi-\phi^r)\\
\lesssim& \|g-g^r\|_{L^{\frac{p^\prime(n-1)}{n(p^\prime-1)}}(\p B_R)}\|\phi-\phi^r\|_{L^{\frac{p^\prime(n-1)}{n-p^\prime}}(\p B_R)}
\end{align*}
By a standard trace estimate and Poincar\'{e}'s inequality
\begin{align*}
\|\phi-\phi^r\|_{L^{\frac{p^\prime(n-1)}{n-p^\prime}}(\p B_R)}\lesssim\|\phi-\phi^r\|_{W^{1,p^\prime}(B_R)}\lesssim \|D\phi-\D\phi^r\|_{L^{p^\prime}(B_R)}.
\end{align*}
Note $\frac{p^\prime(n-1)}{n-p^\prime} =\frac{(n-1)p} n$. By standard properties of convolution, 
\begin{align*}
\|g-g^r\|_{L^\frac{p(n-1)}{n}(\p B_R)}\lesssim r^\frac 1 {n-1} \|g\|_{L^p(\p B_R)}.
\end{align*} If $p\leq 2$, since $V_{p^\prime}(\D\phi,\D\phi^r)\geq \lvert \D\phi-\D\phi^r\rvert^{p^\prime}$, combining estimates and re-arranging concludes the proof. If $p^\prime\leq 2$, we apply H\"older's inequality to see
\begin{align*}
\int_{B_R} V_{p^\prime}(\D\phi,\D\phi^r) \geq \|\D\phi-\D\phi^r\|_{L^{p^\prime}(B_R)}^2(\|\lvert \D\phi\rvert+\lvert \D\phi^r\rvert\|_{L^{p^\prime}(B_R)}^{p^\prime-2}.
\end{align*}
Since due to \eqref{eq:energy} and standard properties of convolution,
\begin{align*}
\||\lvert \D\phi\rvert+\lvert \D\phi^r\rvert\|_{L^{p^\prime}(B_R)}\lesssim \|D\phi\|_{L^{p^\prime}(B_R)}+\|\D\phi^r\|_{L^{p^\prime}(B_R)}\lesssim \|g\|_{L^p(\p B_R)}^{p-1},
\end{align*}
this again gives \eqref{eq:diff}.

\eqref{eq:regularityPhiR} follows from \cite{Lieberman1988} and \cite{Ladyzhenskaya1968}. Again the proof is involved and we don't comment on it here.

Finally, we note by direct calculation using \eqref{ass:Cgrowth}, \eqref{eq:CgrowthDual}, \eqref{eq:c1growthDual} and \eqref{eq:sizeDualGradient} that 
\begin{align*}
[c^\ast(\D\phi)+c(\nabla c^\ast(\D\phi))]_{C^{0,\beta}(B)} \lesssim \|\D\phi\|_{L^\infty(B)}^{p'-1}[D\phi]_{C^{0,\beta}(B)}.
\end{align*}
This concludes the proof.
\end{proof}

\section{Preliminaries}\label{sec:prelim}

\subsection{General notation}
Throughout, we let $1<p<\infty$. $B_r(x)$ will denote a ball of radius $r>0$ centered at $x\in \R^d$. We further write $B_r=B_r(0)$ and $B=B_1(0)$. $c$ denotes a generic constant that may change from line to line. Relevant dependencies on $\Lambda$, say, will be denoted $c(\Lambda)$. We say $a\lesssim b$ and $a\gtrsim b$, if there exists a constant $c>0$ depending only on $d$, $p$ and $\Lambda$ such that $a\leq c b$ and $a\geq c b$, respectively.

 Given $\Omega\subset\R^d$, we denote by $[\cdot]_{C^{0,\alpha}}$, the $\alpha$-H\"older-seminorm. Given $\alpha\in(0,\infty)$, $L^p(\Omega)$ and $W^{\alpha,p}(\Omega)$ denote the usual Lebesgue and (fractional) Sobolev spaces. If $\mu$ is a measure on $\R^d$, $\mu\llcorner \Omega$ denotes its restriction to $\Omega$. We say a couple of measures $\lambda$, $\mu$ on $\R^d$ is admissible if $\lambda,\mu$ are non-negative finite measures satisfying $\lambda(\R^d)=\mu(\R^d)$.

Given $R>0$, we let $\Pi_R(x)=R\frac x {|x|}$ be the projection onto $\p B_R$ and define for every measure $\rho$ on $\R^d$ the projected measure on $\p B_R$, $\hat \rho = \Pi_R \# \rho$, i.e.
\begin{align}\label{eq:radial}
\int \xi \d\hat\rho = \int \xi\left(R\frac x {|x|}\right)\d\rho(x).
\end{align}

A set $\Omega\subset\R^d\times \R^d$ is said to be $c$-cyclically monotone if for any $N\in \N$ and any points $(x_1,y_1),\ldots,(x_N,y_N)\in\Omega$, there holds
$$
\sum_{i=1}^N c(x_i-y_i)\leq \sum_{i=1}^N c(x_i-y_{i+1}),
$$
where we identify $y_{N+1}=y_1$.

A function $f\colon \R^d\to \R$ is called $c$-concave if there exists a function $g\colon \R^d\to \R$ such that 
$$
f(x)=\inf_{y\in \R^d} c(x-y)-g(y)
$$
for all $x\in \R^d$.

\subsection{Optimal transportation}
We recall some definitions and facts about optimal transportation, see \cite{Villani2009} for more details. For this subsection, the full strength of our assumptions \eqref{ass:elliptic}-\eqref{eq:controlledGrowth} is not needed. In fact, assuming that $c$ is lower semi-continuous, convex and satisfies $p$-growth \eqref{ass:growth} would be sufficient in this subsection.

 Given a measure $\pi$ on $\R^d\times\R^d$ we denote its marginals by $\pi_1$ and $\pi_2$ respectively. The set of measures on $\R^d\times \R^d$ with marginals $\pi_1$ and $\pi_2$ is denoted $\Pi(\pi_1,\pi_2)$. Given two positive measures with compact support and equal mass $\lambda$ and $\mu$ we define
$$
W_c(\lambda,\mu)=\min_{\pi_1 = \lambda,\pi_2=\mu} \int c(x-y)\d\pi.
$$
While our notation is reminiscent of the Wasserstein distance, and in fact gives the ($p$-th power of the) Wasserstein $p$-distance in the case  ${c(x-y) = \lvert x-y\rvert^p}$, in general it is not a distance on measures.
Under our hypothesis, an optimal coupling always exists and moreover a coupling $\pi$ is optimal, if and only if its support is $c$-cyclical monotone. 

Moreover, we note the following triangle-type inequality:
\begin{lemma}\label{lem:triangleInequality}
Let $\e\in(0,1)$. There is $C(\e)>0$ such that for any admissible measures $\mu_1, \mu_2, \mu_3$ it holds that
\begin{align*}
W_c(\mu_1,\mu_3)\leq (1+\e) W_c(\mu_1,\mu_2)+C(\e) W_c(\mu_2,\mu_3).
\end{align*}
\end{lemma}
\begin{proof}
Due to the gluing lemma, see e.g. \cite[Lemma 5.5.]{Santambrogio2015}, there exists $\sigma$, a positive measure on $\R^d\times \R^d\times \R^d$ with marginal $\pi_1$ on the first two variables and marginal $\pi_2$ on the last two variables. Here $\pi_1$ and $\pi_2$ are the optimal couplings between $\mu_1$ and $\mu_2$ and $\mu_2$ and $\mu_3$, respectively, with respect to $W_c$. Set $\gamma$ to be the marginal of $\sigma$ with respect to the first and third variable. Then $\gamma\in \Pi(\mu_1,\mu_3)$. It follows using the convexity of $c$ and the triangle inequality in $L^p(\gamma)$ that for any $t\in(0,1)$,
\begin{align*}
W_c(\mu_1,\mu_3)\leq& \left(\int c(x-z)\d\gamma\right)^\frac 1 p\leq \left(\int t c\left(\frac{x-y} t\right)+(1-t)\left(\frac{y-z}{1-t}\right)\d\gamma\right)^\frac 1 p\\
\leq& \left(\int \left( \left(t c\left(\frac{x-y} t\right)\right)^\frac 1 p+\left((1-t)\left(\frac{y-z}{1-t}\right)\right)^\frac 1 p\right)^p\d\gamma\right)^\frac 1 p\\
\leq& \left(\int t c\left(\frac{x-y} t\right)\d\gamma\right)^\frac 1 p+\left(\int (1-t) c\left(\frac{y-z}{1-t}\right)\d\gamma\right)^\frac 1 p.
\end{align*}
Using \eqref{ass:growth} and recalling the definition of $\gamma$, we deduce
\begin{align*}
W_c(\mu_1,\mu_3)\leq \left(\Lambda^2 t^{1-p}\right)^\frac 1 p W_c(\mu_1,\mu_2)+\left(\Lambda^2 (1-t)^{1-p}\right)^\frac 1 pW_c(\mu_2,\mu_3).
\end{align*}
Choosing $t$ sufficiently close to $1$, this gives the desired estimate.
\end{proof}

We require also the following consequence of Lemma \ref{lem:triangleInequality}.
\begin{corollary}\label{cor:addConstant}
Let $\mu_1,\mu_2$ be admissible measures. Then
\begin{align*}
W_c(\mu_1,\mu_2)\lesssim W_c(\mu_1+\mu_2,2\mu_2).
\end{align*}
\end{corollary}
\begin{proof}
Using Lemma \ref{lem:triangleInequality} and sub-additivity of $W_c$, we note for any $\delta>0$,
\begin{align*}
W_c(\mu_1,\mu_2)\leq& (1+\delta) W_c\left(\mu_1,\frac 1 2 (\mu_1+\mu_2)\right)+C(\delta)W_c\left(\frac 1 2 (\mu_1+\mu_2),\mu_2\right)\\
=& (1+\delta) W_c\left(\frac 1 2 \mu_1,\frac 1 2 \mu_2\right)+c(\delta)W_c\left(\frac 1 2 (\mu_1+\mu_2),\mu_2\right)\\
\leq& \frac{1+\delta} 2 W_c\left(\mu_1,\mu_2\right)+C(\delta)W_c\left(\mu_1+\mu_2,2\mu_2\right).
\end{align*}
Re-arranging gives the result.
\end{proof}

We remark that the Benamou-Brenier formula \eqref{eq:eulerianFormulation} needs to be interpreted via duality in general, that is we set
\begin{align}\label{eq:BenamouBrenierDefinition}
\int c\left(\frac{\d j_t}{\d\rho_t}\right)\d \rho_t = \sup_{\zeta\in C^0_c(\R^d)} \left\{\int\zeta \d j_t - \int c^\ast(\xi)\d \rho_t\right\}.
\end{align}

Given $O\subset \R^d$ set $\kappa_{\mu,O}$ to be the generic constant such that $W_{c}(\mu\llcorner O,\kappa_{\mu,O}\d x \llcorner O)$ is well-defined, that is $\kappa_{\mu,O} = \frac{\mu(O)}{\lvert O\rvert}$. If $O=B_R$, we write $\kappa_{\mu,R} = \kappa_{\mu,B_R}$.

It will be convenient to denote $\#_R = (B_R\times \R^d)\cup (\R^d\times B_R)$. We recall the definition of the quantities that we use to measure smallness:
\begin{gather*}
E(R):= \frac 1 {\lvert B_R\rvert} \int_{\#_R} c(x-y)\d \pi,\\
D(R):= \frac 1 {\lvert B_R\rvert} W_p^p(\lambda\llcorner B_R,\kappa_{\lambda,R}\d x\llcorner B_R)+\frac{R^p}{\kappa_{\lambda,R}^{p-1}}(\kappa_{\lambda,R}-1)^p \\
\qquad\qquad+\frac 1 {\lvert B_R\rvert} W_p^p(\mu\llcorner B_R,\kappa_{\mu,R}\d x\llcorner B_R)+\frac{R^p}{\kappa_{\mu,R}^{p-1}}(\kappa_{\mu,R}-1)^p.
\end{gather*}

We will find it convenient to work with trajectories $X(t) = t x+(1-t)y$. In this context, it is useful to work on the domain
\begin{align*}
\Omega_R = \{(x,y)\in \#_3\colon \exists t\in[0,1] \text{ s.t. } X(t)\in \overline B_R\}.
\end{align*}
To every trajectory $X\in\Omega$, we associate entering and exiting times of $B_R$:
\begin{gather*}
\sigma_R:= \min\{t\in[0,1]\colon X(t)\in \overline B_R\}\\
\tau_R:= \max\{t\in[0,1]\colon X(t)\in \overline B_R\}.
\end{gather*}
Often, we will drop the subscripts and denote $\Omega = \Omega_R$, $\sigma=\sigma_R$ and $\tau = \tau_R$.
Further, we will need to track trajectories entering and leaving $B_R$. This is achieved through the non-negative measures $f_R$ and $g_R$ concentrated on $\p B_R$ and defined by the relations
\begin{gather}\label{eq:2}
\int\zeta \d f_R = \int_{\Omega \cap \{X(\sigma)\in \p B_R\}} \zeta(X(\sigma))\d \pi,\nonumber\\
\int\zeta \d g_R = \int_{\Omega \cap \{X(\tau)\in \p B_R\}} \zeta(X(\sigma))\d \pi.
\end{gather}
Note that the set of trajectories $\Omega \cap \{X(\sigma)\in \p B_R\}$ implicitly defines a Borel measurable subset of $\R^d\times \R^d$, namely the pre-image under the mapping $(x,y)\to X$, which is continuous from $\R^d\times \R^d$ into $C^0([0,1])$. Thus, the integrals in \eqref{eq:2} are well-defined. We will often use similar observations without further justification.

\subsection{Estimating radial projections}
We record a technical estimate concerning radial projections we will require.
\begin{lemma}\label{lem:projection}
For $R>0$, there exists $1\geq \e(d)>0$ such that for every $g\geq 0$ with $\mathrm{Spt}\, g\subset B_{(1+\e)R}\setminus B_{(1-\e)R}$ we have
\begin{align*}
R^{1-d}\left(\int g\right)^p\lesssim \int_{\p B_R} \hat g^p\lesssim \sup g^{p-1} \int |R-|x||^{p-1}g
\end{align*}
Here $\hat g$ is the radial projection of $g$ defined in \eqref{eq:radial}.
\end{lemma}
\begin{proof}
By scaling we may assume $R=\sup g =1$. The first inequality is then a direct consequence of Jensen's inequality.

For the second inequality, note that if $\e\ll 1$, $\sup_{\p B_1} |\hat g|\ll 1$, since we assume ${\mathrm{Spt}\, g\subset B_{1+\e}\setminus B_{1-\e}}$. Fix $\omega\in \p B_1$ and set $\psi(r)=r^{d-1}g(r\omega)$ for $r>0$. Then we have ${0\leq \psi\leq (1+\e)^{d-1}\leq 2}$ and
$$
\int_0^\infty\psi = \hat g(\omega).
$$
We conclude that for $\omega\in \p B_1$,
$$
\int_0^\infty |1-r|^{p-1} r^{d-1}g(r\omega)\geq \min_{0\leq \tilde\psi\leq 2, \int\tilde\psi = \hat g(\omega)} \int_0^\infty |1-r|^{p-1} \tilde \psi(r)\gtrsim \hat g(\omega).
$$
The last inequality holds, since the minimiser of
$$
\min_{0\leq \tilde\psi\leq 2, \int\tilde\psi = \hat g(\omega)} \int_0^\infty |1-r|^{p-1} \tilde \psi(r)
$$
is given by $2I\left(|r-1|\leq \frac 1 4 \hat g(\omega)\right)$.
\end{proof}

\section{A \texorpdfstring{$L^\infty$}{}-bound on the displacement}\label{sec:Linfty}
A key point in our proof will be that trajectories do not move very much. Since we assume $E(4)\ll 1$, this is evidently true on average. However, we will require to control the length of trajectories not just on average, but in a pointwise sense. We establish this result in this section. In the quadratic case, the proof in \cite{Goldman2018} relies on the fact that $2$-monotonicity is equivalent to standard monotonicity. In our setting this is not available and we hence provide a different proof. Our proof heavily relies on the strong $p$-convexity of $c$.

\begin{lemma}\label{lem:linfty}
Let $1<p<\infty$.
Let $\pi$ be a coupling between two admissible measures $\lambda$ and $\mu$. Assume that $\mathrm{Spt}\,\pi $ is cyclically monotone with respect to $c$-cost and that ${E(4)+D(4)\ll 1}$.
Then for every $(x,y)\in \mathrm{Spt}\,\pi\cap \#_3$,
we have
\begin{align}\label{eq:main}
\lvert x-y\rvert\lesssim \left(E(4)+D(4)\right)^\frac 1 {p+d}.
\end{align}
As a consequence, for $(x,y)\in \mathrm{Spt}\,\pi$ and $t\in[0,1]$,
\begin{align}\label{eq:bound2}
x\in B_{3} \text{ or } y\in B_3 \Rightarrow (1-t)x+ty\in B_{4}.
\end{align}
\end{lemma}

In the proof of Lemma \ref{lem:linfty} we require the following technical result, which we state independently as we will require it again in the future.
\begin{lemma}\label{lem:C2measures}
Let $1<p<\infty$ and $0< \alpha<1$.
For every $R>0$, $\xi\in C^{0,\alpha}(B_R)$ and $\mu$ supported in $B_R$ with $\mu(B_R)\sim \lvert B_R\rvert$,
\begin{align}\label{eq:taylor}
\left\lvert\int_{B_R}\xi(\d\mu-\kappa_{\mu,R}\d x)\right\rvert\leq& [\xi]_{C^{0,\alpha}(B_R)} W_{c}(\mu,\kappa_{\mu,R}\d x \llcorner B_R)^\frac \alpha p R^\frac{2d(p-\alpha)} p.
\end{align}
In case $\alpha =1$, \eqref{eq:taylor} holds with $C^{0,1}$ replaced by $C^1$. Further, if in addition we have ${\xi\in C^{\lfloor{p-1}\rfloor,p-\lfloor{p-1}\rfloor}(B_R)}$, there is $C>0$ such that 
\begin{align*}
\left \lvert \int_{B_R} \xi(\d\mu-\kappa_{\mu,R}\d x)\right\rvert \leq& C\sum_{i=1}^{\lfloor{p-1}\rfloor}\left(\kappa_{\mu,R}\int \lvert \D^i\xi\rvert^\frac{p}{p-i}\d x\right)^\frac{p-i} {p} W_c(\mu,\kappa_{\mu,R}\d x \llcorner B_R)^\frac i p\\
&\quad + [\xi]_{C^{\lfloor{p-1}\rfloor,p-\lfloor{p-1}\rfloor}} W_c(\mu,\kappa_{\mu,R}\ x \llcorner B_R).
\end{align*}
\end{lemma}

\begin{proof}
Integrate the estimate
$$
|\xi(x)-\xi(y)|\leq [\xi]_{C^{0,\alpha}}|x-y|^\alpha
$$
against an optimal transport plan $\pi$ between $\mu$ and $\kappa_{\mu,R}\d x\llcorner B_R$ to find, 
\begin{align*}
\left|\int_{B_R}\xi(\d\mu-\kappa_\mu\d x)\right|\leq& [\xi]_{C^{0,\alpha}(B_R)} \int_{B_R}|x-y|^\alpha \d \pi.
\end{align*}
Applying H\"older and using \eqref{ass:growth} the result follows.

To obtain the second estimate, we proceed similarly, but start with the estimate
\begin{align*}
\lvert\xi(x)-\xi(y)-\sum_{\lvert \alpha\rvert=1}^{\lfloor{p-1}\rfloor}\D^\alpha \xi(y)\frac{(x-y)^\alpha}{\lvert \alpha\rvert!}\rvert \leq [\xi]_{C^{\lfloor{p-1}\rfloor,p-\lfloor{p-1}\rfloor}(B_R)} \lvert x-y\rvert^{p}.
\end{align*}
The result follows using \eqref{ass:growth} and using H\"older to estimate
\begin{align*}
\int \lvert \D^\alpha \xi(y)\rvert \lvert x-y\rvert^{\lvert \alpha\rvert} \d\pi \leq \left(\kappa_{\mu,R}\int \lvert \D^{\lvert \alpha\rvert}\xi\rvert^\frac{p}{p-\lvert \alpha\rvert}\d x\right)^\frac{p-\lvert \alpha\rvert} p W_c(\mu,\kappa_{\mu,R}\d x \llcorner B_R)^\frac {\lvert \alpha\rvert} p.
\end{align*}
\end{proof}

We proceed to prove Lemma \ref{lem:linfty}.
\begin{proof}[Proof of Lemma \ref{lem:linfty}]
Fix $(x,y)\in \mathrm{Spt}\,\pi\cap \#_3$.  Without loss of generality we may assume that $(x,y)\in B_{3}\times \R^d$.

\textbf{Step 1. Barrier points exist in all directions:} In this step we show that in all directions we may find points $(x',y')\in \mathrm{Spt}\,\pi$ with $x'\approx y'$. To be precise, consider an arbitrary unit vector $\bm n\in \R^d$ and let $r>0$. We show that for any $\bm n$, and all $r\ll 1$, there is $M=M(p,d,\Lambda)>0$ and $(x',y')\in \mathrm{Spt}\,\pi\cap (B_r(x+2r \bm n)\times \R^d)$ such that
\begin{align*}
c(x'-y')\leq \frac {ME(4)} {r^d}.
\end{align*}

Assume, for contradiction, that for any $M>0$, there is $\bm n\in \R^d$ and $r>0$ such that for all $(x',y')\in \mathrm{Spt}\,\pi\cap (B_r(x+2r \bm n)\times \R^d)$, $|x'-y'|\geq \frac {M E(4)}{r^d}$. 
 Let $\eta$ be a non-negative, smooth cut-off supported in $B_r(x+2r \bm n)$ satisfying 
 $$\sum_{i=1}^{\lfloor{p-1}\rfloor} r^i\sup|\D^i\eta|+r^{p}[\eta]_{C^{\lfloor{p-1}\rfloor,p-\lfloor{p-1}\rfloor}}\lesssim 1.
 $$
 Then
\begin{align*}
E(4)\gtrsim \int \int \eta(x)c(x-y)\d\pi(x,y)\geq \int \int \frac{M E(4)}{r^d}\eta(x)\d\pi(x,y)= \frac{M E(4)}{r^d}\int \eta(x) \d\mu(x).
\end{align*}
However, due to Lemma \ref{lem:C2measures} and noting $\kappa_{\mu,4}\sim 1$,
\begin{align*}
\left|\int \eta\d\mu(x)-\kappa_{\mu,4}\int \eta\d x\right|\lesssim \sum_{i=1}^{\lfloor{p-1}\rfloor} r^{\frac {d(p-i)}{p}-i}D(4)^\frac i p+r^{-p} D(4).
\end{align*}
Normalising $\eta$ such that $\int_{B_r(x+2r n)}\eta\d x\sim r^d$, we can guarantee $\kappa_{\mu,4}\int \eta\d x\sim \kappa_{\mu,4} r^d\sim r^d$.
 Ensuring $D(4)\ll r^{p+d}$, so that $\sum_{i=1}^{\lfloor{p-1}\rfloor} r^{\frac {d(p-i)}{p}-i}D(4)^\frac i p+r^{-p} D(4) \ll r^d$, we may thus conclude
\begin{align*}
E(4)\gtrsim \frac{M E(4)}{r^d} r^d = M E(4).
\end{align*}
As $M$ was arbitrary, this is a contradiction.

\textbf{Step 2. Building barriers:} In this step, we show that if we are given points $${(x',y')\in \mathrm{Spt}\,\pi\cap (B_r(x+2r \bm n)\times \R^d)}$$ such that $|x'-y'|\leq \frac {ME(4)} {r^d}$ for some $M=M(p,d,\Lambda)>0$, then there is a cone $C_{x,x'}$ with vertex $x'+r\rho(x'-x)$ for some $\rho=\rho(p,d,\Lambda)>0$, aperture $\alpha=\alpha(p,d,\Lambda)$ and axis $x'-x$ such that $y\not\in C_{x,x'}$.

Without loss of generality, we may assume that $x'-x$ points in the $e_n$ direction. Moreover, considering the cost $c(\cdot)-c(x)$, we may assume that $c(x)=0$. 
Suppose for a contradiction that
$$
y\in C_{x,x'}=x'+\{a\in \R^{d-1}\times \R^+\colon d(a,\Gamma)\leq \alpha (|\overline a-x'|-r \rho)\}
$$
for some $\alpha,\rho>0$ to be determined. Here $\Gamma= \{t(x'-x)\colon t\geq 0\}$ and $\overline a$ denotes the orthogonal projection of a point $a\in \R^{d-1}\times \R^+$ onto $\Gamma$. We want to show that then
\begin{align}\label{eq:contraMonoton}
c(x-y)\geq c(x'-y)+c(x-y').
\end{align}
\eqref{eq:contraMonoton} is a contradiction to the $c$-monotonicity of $\pi$ and hence proves the stated claim.

We note that we may assume $x=0$. Indeed, setting $z = y-x$, $z'=y'-x$ and $\tilde z=x'-x$, \eqref{eq:contraMonoton} becomes
\begin{align*}
c(-z)\geq c(\tilde z-z)+c(-z')
\end{align*}
with $|\tilde z|\leq \frac{M E}{r^d}$ and $z\in C_{0,\tilde z}$, which we recognise as precisely the situation we are in if $x=0$.

\definecolor{uququq}{rgb}{0.25,0.25,0.25}
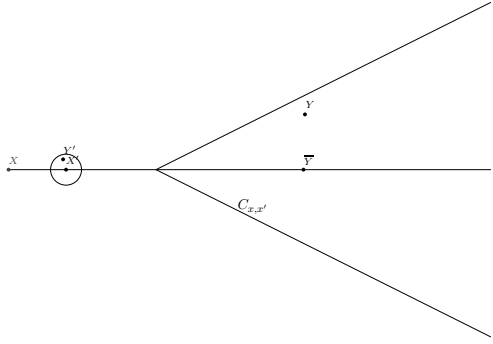
\begin{figure}[ht]
\centering
\resizebox{0.5\textwidth}{!}{
\begin{tikzpicture}[line cap=round,line join=round,>=triangle 45,x=1.0cm,y=1.0cm]
\clip(-1.66,-5.3) rectangle (13.08,7.48);
\draw [domain=4.0:13.080000000000005] plot(\x,{(-4--1*\x)/2});
\draw [domain=4.0:13.080000000000005] plot(\x,{(--4-1*\x)/2});
\draw(1.56,0) circle (0.42cm);
\draw (6.08,-0.68) node[anchor=north west] {$C_{x,x'}$};
\draw [domain=0.0:13.080000000000005] plot(\x,{(-0-0*\x)/8});
\begin{scriptsize}
\fill [color=uququq] (0,0) circle (1.5pt);
\draw[color=uququq] (0.14,0.26) node {$X$};
\fill [color=black] (8.04,1.5) circle (1.5pt);
\draw[color=black] (8.18,1.76) node {$Y$};
\fill [color=black] (8,0) circle (1.5pt);
\draw[color=black] (8.14,0.26) node {$\overline{Y}$};
\fill [color=black] (1.56,0) circle (1.5pt);
\draw[color=black] (1.74,0.26) node {$X'$};
\fill [color=black] (1.48,0.28) circle (1.5pt);
\draw[color=black] (1.66,0.54) node {$Y'$};
\end{scriptsize}
\end{tikzpicture}
}
\caption{Geometric situation in Step 2.}
\end{figure}

Taking $\rho\geq 4$, we then estimate using the growth assumption \eqref{ass:Cgrowth} and the strict convexity assumption \eqref{ass:elliptic}
\begin{align*}
c(-y)\geq& c(-\overline y)+c(x'-y')-\Lambda U(-y,-\overline y)\\
\geq& \frac{|\overline y|}{|-\overline y+x'|}c(-\overline y+x')+\frac{\lambda|x'|}{|\overline y|}V(-\overline y,0)-\Lambda U(-y,-\overline y)\\
\geq& c(-\overline y+x')+c(x')+\frac{\lambda|x'|}{|\overline y|}V(-\overline y,0)+\frac{\lambda|\overline y-2x'|}{|\overline y-x'|}V(-\overline y+x',0)-\Lambda U(-y,-\overline y)\\
\geq& c(-y+x')+c(-x')-\Lambda U(-y+x',-\overline y+x')+\frac{\lambda|x'|}{|\overline y|}V(-\overline y,0)\\
&\quad+\frac{\lambda|\overline y-2x'|}{|\overline y-x'|}V(-\overline y+x',0)-\Lambda U(-y,-\overline y)\\
\geq& c(-y+x')+c(-y')-\Lambda U(-y+x',-\overline y+x')+\frac{\lambda|x'|}{|\overline y|}V(-\overline y,0)\\
&\quad+\frac{\lambda|\overline y-2x'|}{|\overline y-x'|}V(-\overline y+x',0)-\Lambda U(-y,-\overline y)-\Lambda U(-x',-y')\\
\end{align*}

In particular, it suffices to show
\begin{align}\label{eq:sufficient}
&c(x'-y')+ \frac{\lambda|x'|}{|\overline y|}V(-\overline y,0)+\frac{\lambda|\overline y-2x'|}{|\overline y-x'|}V(-\overline y+x',0)\nonumber\\
\geq& \Lambda U(-y,-\overline y)+\Lambda U(-y+x',-\overline y+x').
\end{align}
We note that,  if $\rho\geq 8$, $|\overline y-2x'|\geq \frac 1 2|\overline y|$. Then we can estimate
\begin{align*}
\frac{\lambda|x'|}{|\overline y|}V(-\overline y,0)+\frac{\lambda|\overline y-2x'|}{|\overline y-x'|}V(-\overline y+x',0)
=& \lambda(|x'| |\overline y|^{p-1}+|\overline y-2 x'| |\overline y-x'|^{p-1})\\
\gtrsim &  |\overline y|^{p}.
\end{align*}
Further, if $E(4)\leq \e r^{d+1}$,
\begin{align*}
&\Lambda U(-y,-\overline y)+\Lambda U(-y+x',-\overline y+x')+\Lambda U(-x',-y')\\
\leq& 2\Lambda (2|y|)^{p-1}|y-\overline y|+\Lambda \left(|x'|+|y'|\right)^{p-1} |x'-y'|\\
\lesssim& |y|^{p-1}\alpha |\overline y-x'|+\frac{M E}{r^d} (2r)^{p-1}\\
\lesssim& \alpha|\overline y|^p+\e |\overline y|^p .
\end{align*}
Thus choosing $\alpha$, $\e>0$, sufficiently small, we find \eqref{eq:sufficient} holds, proving our claim.

\textbf{Step 3. Proving the $L^\infty$-bounds:}
Choose $r=c (E(4)+D(4))^{\frac 1 {p+d}}$. For sufficiently large choice of $c>0$ and selecting $c(d)$ directions $n_i$, by Step 1 and Step 2 we obtain points 
$${(x',y')\in \mathrm{Spt}\,\pi\cap (B_r(x+2 r n_i)\times \R^d)}$$
and cones $(C_{x,x_i'})_{i\leq c(d)}$ with vertices $x_i'+r\rho(x_i'-x)$, aperture $\alpha$ and axis $x_i'-x$ such that for some $c(\alpha)>0$,
$${y\not\in \cup C_{y_i}} \text{ and }\R^d\setminus B_{(\rho+c(\alpha))r}(x)\subset\cup C_{y_i}.
$$ In particular, we have
\begin{align*}
|y-x|\leq (\rho+c(\alpha))r\lesssim (E+D)^\frac 1 {p+d},
\end{align*}
that is \eqref{eq:main}.

\eqref{eq:bound2} is a direct consequence of \eqref{eq:main}, concluding the proof.
\end{proof}

We record two consequences of Lemma \ref{lem:linfty} we will use later.
\begin{corollary}\label{cor:linfty}
Under the assumptions of Lemma \ref{lem:linfty}, it holds that
\begin{gather*}
\int_2^3 \int_{\Omega \cap \{\exists t\in[0,1]\colon X(t)\in \p B_R)\}} c(x-y)\d \pi\d R\lesssim (E(4)+D(4))^{1+\frac 1 {p+d}},\\
\int_2^3 \int_{\Omega} I(\{\exists t\in[0,1]\colon X(t)\in \p B_R)\})\d \pi\d R \lesssim (E(4)+D(4))^\frac 1 {p+d}.
\end{gather*}
\end{corollary}
\begin{proof}
We use Lemma \ref{lem:linfty} to deduce there is $C>0$ such that
\begin{align*}
&\int_2^3 \int_{\Omega \cap \{\exists t\in[0,1]\colon X(t)\in \p B_R)} c(x-y)\d \pi\d R\\
\leq& \int_2^3 \int_{(B_{7/2}\setminus B_{3/2})\times (B_{7/2}-B_{3/2})}I(\{\lvert \lvert x\rvert-R\rvert\leq C(E(4)+D(4))^\frac 1 {p+d}\}) c(x-y)\d \pi\d R\\
\lesssim& (E(4)+D(4))^\frac 1 {p+d} \int_{\#_4} c(x-y)\d \pi\\
\lesssim& (E(4)+D(4))^{1+\frac 1 {p+d}}.
\end{align*}

Further, again using Lemma \ref{lem:linfty}, there is $C>0$ such that,
\begin{align*}
&\int_2^3 \int_\Omega I(\{\exists t\in[0,1]\colon X(t)\in \p B_R\})\d \pi\d R\\
\leq& \int_2^3 \pi(\Omega \cap \{\lvert X(0)-R\rvert \leq C(E(4)+D(4))^\frac 1 {p+d}\})\d R\\
=& \int_2^3 \mu(\{\lvert \lvert x\rvert-R\rvert \leq C(E(4)+D(4))^\frac 1 {p+d})\d R\\
\leq& \int_{B_{7/2}\setminus B_{3/2}\times B_{7/2}\setminus B_{3/2}} \int I(\lvert \lvert x\rvert -R\rvert \leq C(E(4)+D(4))^\frac 1 {p+d})\d R\d \mu\\
\lesssim& (E(4)+D(4))^\frac 1 {p+d} \mu(B_4)\\
\lesssim& (E(4)+D(4))^\frac 1 {p+d}.
\end{align*}

\end{proof}

\section{A localisation result}\label{sec:localisation}
In order to prove Theorem \ref{thm:main}, we need to use optimality in a localised way, as the quantity we need to estimate is a local quantity. In general, given a minimiser $\pi$ of optimal transport with cost function $c$ between two measures $\lambda$ and $\mu$, it is not true that the localised transport cost of $\pi$ is approximately equal to the optimal transport cost between the localised measures $\lambda \llcorner B_R$ and $\mu\llcorner B_R$. In other words, it is in general not the case that
\begin{align*}
\int_{\Omega} c(x-y)\d\pi\approx W_c(\lambda\llcorner B_R,\mu\llcorner B_R).
\end{align*}
However, if we take into account the entry points of trajectories entering $B_R$ (which we denoted $f_R$, c.f. \eqref{eq:2}) and the exit points of trajectories exiting $B_R$ (which we denoted $g_R$, c.f. \eqref{eq:2}), the values are close as we show in the next lemma.
\begin{lemma}\label{lem:localisation}
Let $\lambda,\mu$ be admissible measures. Suppose $\pi\in \Pi(\lambda,\mu)$ minimises \eqref{eq:problem}. Let $R\in[2,3]$ and define $f_R,g_R$ as in \eqref{eq:2}. Then for any $\tau,\delta>0$, there is $\e>0$ such that if $E(4)+D(4)\leq \e$, then
\begin{align*}
\int_\Omega c(x-y)\d \pi\leq (1+\delta) W_c(\lambda \llcorner B_R+f_R,\mu\llcorner B_R+ g_R)+\tau \left(E(4) + D(4)\right)
\end{align*}
\end{lemma}
\begin{proof}
Introduce the weakly continuous family of probability measures $\{\lambda_z\}_{z\in \p B_R}$ such that
\begin{align*}
\int_{\Omega \cap \{X(\sigma)\in \p B_R\}} \zeta(x,X(\sigma))\pi(\d x\d y) = \int_{\p B_R}\int \zeta(x,z)\lambda_z(\d x)f_R(\d z)
\end{align*}
for any test function $\zeta$ on $\R^d\times \R^d$. Likewise, introduce $\{\mu_w\}_{w\in \p B_R}$ via
\begin{align*}
\int_{\Omega \cap \{X(\tau)\in \p B_R\}}\zeta(X(\tau),y)\pi(\d x\d y) = \int_{\p B_R}\int \zeta(w,y)\mu_w(\d y)g_R(\d w).
\end{align*}
Let $\overline\pi$ be an optimal plan for $W_c(\lambda\llcorner B_R+f_R,\mu\llcorner B_R+g_R)$. Define a competitor $\tilde \pi$ for $\pi$ by requiring the following formula to hold for any test function $\zeta$ on $\R^d\times \R^d$,
\begin{align}\label{eq:tildepi}
&\int \zeta(x,y)\tilde\pi(\d x \d y) \nonumber\\
=& \int_{\Omega^c}\zeta(x,y)\d \pi(x,y) + \int_{B_R\times B_R} \zeta(x,y)\overline\pi(\d x\d y)+ \int_{\p B_R\times B_R}\int \zeta(x,y)\lambda_z(\d x)\overline\pi(\d z \d y)\nonumber\\
&+ \int_{B_R\times \p B_R}\zeta(x,y)\mu_w(\d y)\overline\pi(\d x \d w) + \int_{\p B_R\times \p B_R} \int\int \zeta(x,y)\mu_w(\d y)\lambda_z(\d x)\overline\pi(\d z \d w)\nonumber\\
=& I + II + III + IV + V.
\end{align}

In order to see that $\tilde \pi\in \Pi(\lambda,\mu)$, by symmetry it suffices to check that the first marginal is $\lambda$. Hence test \eqref{eq:tildepi} against $\zeta(x)$. We begin by noting that due to the definition of $\mu_w$ and using that $\overline\pi$ is supported in $\overline B_R$,
\begin{align*}
II + IV  = \int_{B_R\times \R^d}\zeta(x)\overline\pi(\d x\d y) = \int_{B_R}\zeta(x)\mu(\d x) = \int_{\Omega \cap \{X(\sigma)\in B_R)\}} \zeta(x)\pi(\d x \d y).
\end{align*}
Similarly, using also the definition of $f_R$,
\begin{align*}
III + V =& \int_{\p B_R\times \R^d} \int \zeta(x)\lambda_z(\d x)\overline\pi(\d z \d y) = \int_{\p B_R} \zeta(z)f_R(\d z) \\
=& \int_{\Omega \cap \{X(\sigma)\in \p B_R\}} \zeta(x)\pi(\d x \d y).
\end{align*}
In particular, we have shown
\begin{align*}
\int \zeta(x)\tilde \pi(\d x \d y) = \int \zeta(x)\pi(\d x\d y) = \int \zeta(x)\lambda(\d x)
\end{align*}
as desired.

Using optimality of $\pi$ in the form
\begin{align*}
\int_{\Omega} c(x-y)\d \pi + \int_{\Omega^c} c(x-y)\d \pi \leq \int c(x-y)\d\tilde\pi
\end{align*}
and testing \eqref{eq:tildepi} against $\zeta(x,y)= c(x-y)$, we learn
\begin{align*}
&\int_\Omega c(x-y)\d\pi \\
\leq& \int_{B_R\times B_R} c(x-y)\overline\pi(\d x \d y) + \int_{\p B_R\times B_R} \int c(x-y)\lambda_z(\d x)\overline\pi(\d z \d y) \\
&\quad +\int_{B_R\times \p B_R} c(x-y)\mu_w(\d y)\overline\pi(\d x \d w) \\
&\qquad+ \int_{\p B_R\times \p B_R} \int \int c(x-y)\mu_w(\d y)\lambda_z(\d x) \overline\pi(\d z \d w)\\
=& \int_{B_R\times B_R} f_1 \overline\pi(\d x \dy ) + \int_{\p B_R\times B_R} f_2 \overline\pi(\d z\d y) + \int_{B_R\times \p B_R} f_3 \overline \pi(\d x \d w) \\
&\quad+ \int_{\p B_R\times \p B_R} f_4 \overline \pi(\d z \d w)
\end{align*}
As in the proof of Lemma \ref{lem:triangleInequality}, for any $\delta>0$, there is $C_\delta>0$ such that for any $x,y,z$,
\begin{align*}
c(x-z)\leq (1+\delta) c(x-y) + C_\delta c(y-z).
\end{align*}
Using this in combination with the fact that $\lambda_z$, $\mu_w$ are probability measures we deduce
\begin{align*}
f_2 \leq& (1+\delta)c(z-y) + C(\delta)\tilde f_2, \quad f_3 \leq (1+\delta)c(x-w) + C(\delta)\tilde f_3\\
f_4\leq& (1+\delta)c(z-w)+C(\delta) \tilde f_4,
\end{align*}
where
\begin{gather*}
\tilde f_2(z,y) = \int c(x-z)\lambda_z(\d x),\quad \tilde f_3(x,w) = \int c(w-y)\mu_w(\d y),\\
 \tilde f_4(z,w) = \tilde f_2(z,y) + f_3(x,w).
\end{gather*}
In particular, we deduce
\begin{align*}
&\int_\Omega c(x-y)\d \pi \\
\leq& (1+\delta)\int_{\overline B_R\times \overline B_R}c(x-y)\overline\pi(\d x \d y)+ C(\delta)\int_{\p B_R\times B_R} \tilde f_2 \overline\pi(\d z\d y) \\
&\quad+ C(\delta) \int_{B_R\times \p B_R} \tilde f_3\overline\pi(\d x \d w)+ C(\delta) \int_{\p B_R\times \p B_R} \tilde f_4\overline\pi(\d z \d w)\\
=& (1+\delta)\int_{\overline B_R\times \overline B_R}c(x-y)\overline\pi(\d x \d y) + 2 C(\delta)\int_{\p B_R\times \R^d} \int c(x-z)\lambda_z(\d x)\overline\pi(\d z\d y) \\
&\quad + 2C(\delta)\int_{\R^d\times \p B_R} c(w-y)\mu_y(\d y)\overline\pi(\d x \d w)\\
=& (1+\delta)\int_{\overline B_R\times \overline B_R}c(x-y)\overline\pi(\d x \d y) + 2 C(\delta)\int_{\p B_R\times \R^d} \int c(x-z)\lambda_z(\d x)f_R(\d z) \\
&\quad + 2C(\delta)\int_{\R^d\times \p B_R} c(w-y)\mu_y(\d y)g_R(\d w)\\
=& (1+\delta)\int_{\overline B_R\times \overline B_R}c(x-y)\overline\pi(\d x \d y) + 2 C(\delta)\int_{\Omega \cap \{X(\sigma)\in \p B_R\}} c(x-X(\sigma))\d \pi \\
&\quad + 2C(\delta)\int_{\Omega \cap \{ X(\tau) \in \p B_R\}} c(X(\tau)-y)\d \pi\\
\end{align*}
In order to obtain the second to last line we used the admissibility of $\overline\pi$. Now note on the one hand, that due to optimality of $\overline \pi$,
\begin{align*}
\int_{\overline B_R\times \overline B_R} c(x-y)\d\overline\pi = W_c(\lambda \llcorner B_R + f_R,\mu\llcorner B_R + g_R).
\end{align*}
On the other hand, on $\Omega \cap \{X(\sigma)\in \p B_R)\} \cap \{X(\tau)\in \p B_R\}$, for some $\rho_1, \rho_2\geq 0$ with $\rho_1 + \rho_2\leq 1$, due to convexity of $c$ and $c(0)=0$,
\begin{align*}
c(x-X(\sigma)) + c(X(\tau)-y) = c(\rho_1(x-y)) + c(\rho_2(x-y)) \leq c(x-y).
\end{align*}
Thus, we have shown
\begin{align*}
&\int_\Omega c(x-y)\d\pi\\
\leq& (1+\delta) W_c(\lambda \llcorner B_R + f_R,\mu\llcorner B_R+g_R)+C(\delta) \int_{\Omega \cap (\{X(\sigma\in \p B_R\}\cup \{X(\tau)\in \p B_R\})} c(x-y)\d \pi\\
=& (1+\delta) W_c(\lambda \llcorner B_R + f_R,\mu\llcorner B_R+g_R)+C(\delta) \int_{\Omega \cap \{\exists t\in[0,1]\colon X(t)\in \p B_R\}} c(x-y)\d \pi\\
\leq& (1+\delta) W_c(\lambda \llcorner B_R + f_R,\mu\llcorner B_R+g_R)+C(\delta) (E(4)+D(4))^{1+\frac 1 {p+d}}.
\end{align*}
To obtain the last line, we used Corollary \ref{cor:linfty}. Choosing $\e$ sufficiently small the result follows.
\end{proof}

\section{Approximating the boundary data}\label{sec:boundaryData}
Before we can implement the $c^\ast$-harmonic approximation, we face another problem. Lemma \ref{lem:localisation} suggests that the $c^\ast$-harmonic function $\phi$ we should use in Theorem \ref{thm:main} is given as a solution of the following Neumann-problem
\begin{align*}
\begin{cases}
-\Div \nabla c^\ast(\D\phi)= \mu-\lambda \quad& \text{ in } B_R\\
\nabla c^\ast(\D\phi)\cdot \nu = g_R-f_R & \text{ on } \p B_R.
\end{cases}
\end{align*}
However, $f_R$, $g_R$ as well as $\lambda,\mu$ are not sufficiently smooth for $\phi$ to make sense as a weak solution and we will not be able to apply the regularity results of Lemma \ref{lem:ellipticRegularity} as it stands. Hence, we will approximate $f_R$, $g_R$ by suitable $L^p(\p B_R)$-functions $\bar f_R$ and $\bar g_R$ and will replace $\lambda-\mu$ with $c=\int_{\p B_R} g_R-f_R$. After choosing a suitable radius $R\in[2,3]$, this approximation is given by the following result:
\begin{lemma}\label{lem:approximation}
Let $\tau>0$. Let $\lambda,\mu$ be admissible measures on $\R^d$. There is $\e>0$ such that if $E(4)+D(4)\leq \e$, then for every $R\in[2,3]$ there exist non-negative functions $\overline f_R$, $\overline g_R$ such that
\begin{align*}
W_c(f_R,\overline f_R)+W_c(g_R,\overline g_R)\lesssim \tau E(4)+ D(4)\\
\int_2^3 \int_{\p B_R} \overline g_R^p+\overline f_R^p \d R\lesssim E(4)+D(4)
\end{align*}
Here $f_R, g_R$ are the functions defined in \eqref{eq:2}.
\end{lemma}
\begin{proof}
By symmetry it suffices to focus on the terms involving $g$.

We begin by constructing $\overline g_R$. 
Let $\overline\pi$ be optimal for $W_c(\mu\llcorner B_4,\kappa_{\mu,4} \d z\llcorner B_4)$. Extend $\overline \pi$, which is supported on $B_4\times B_4$ by $\frac{\mu\otimes \mu}{\mu(\R^d)}$ to $\R^d\times \R^d$. Note that this extension, still denoted $\overline\pi$ has marginals $\mu$ and $\kappa_{\mu,4}\d z\llcorner B_4+\mu \llcorner B_4^c$. Further due to the definition of $D$, $W_c(\mu\llcorner B_4,\kappa_{\mu,4} \d z \llcorner B_4)\lesssim D(4)$. Introduce the family of probability measures $\{\overline \pi(\cdot|y)\}_{y\in \R^d}$ via asking that for every test function $\zeta$
\begin{align*}
\int \zeta(y,z)\overline \pi(\d z|y)\mu(\d y) = \int \zeta(y,z)\overline \pi(\d y \d z).
\end{align*}
Let $\pi$ be the minimiser of $W_c(\lambda,\mu)$. Then define $\tilde \pi$ on $\R^d\times \R^d\times \R^d$ by the formula
\begin{align*}
\int \zeta(x,y,z)\tilde \pi(\d x \d y \d z) = \int \int \zeta(x,y,z)\overline\pi(\d z|y)\pi(\d x \d y)
\end{align*}
valid for any test function $\zeta$. We note that with respect to the $(x,y)$ variables $\tilde \pi$ has marginal $\pi$, while with respect to the $(y,z)$ variables $\tilde \pi$ has marginal $\overline\pi$. 

Fix $R\in[2,3]$. Extend a trajectory $X\in\Omega$ in a piecewise affine fashion by setting for $t\in[1,2]$,
\begin{align*}
X(t) = (t-1)z+(2-t)y.
\end{align*}
Note that the distribution $g^\prime$ of the endpoint of those trajectories that exit $\overline B_R$ during the time interval $[0,1]$ is given by
\begin{align}\label{eq:1}
\int\zeta \d g^\prime = \int_{\Omega \cap \{X(\tau)\in \p B_R\}} \zeta(z)\tilde \pi(\d x \d y \d z).
\end{align}
Note that due to Lemma \ref{lem:linfty}, $y=X(1)\in B_4$ for any trajectory $X$ that contributes to \eqref{eq:1}. Since $\overline \pi(B_4,B_4^c) = 0$, we deduce that also $z=X(2)\in B_4$ and hence that $g^\prime$ is supported in $B_4$. In particular, we may estimate for any $\zeta\geq 0$, using that the second marginal of $\overline\pi$ is $\kappa_{\mu,4}\d z\llcorner B_4+\mu \llcorner B_4^c$,
\begin{align*}
\int \zeta \d g^\prime \leq \int_{\{z\in B_4\}} \zeta(z)\overline \pi(\d y \d z) = \kappa_{\mu,4}\int_{B_4}\zeta.
\end{align*}
This shows that $g^\prime$ has a density, still denoted $g^\prime$, satisfying $g^\prime \leq \kappa_{\mu,4}$ and allows us to conclude the construction of $\overline g_R$ by defining
\begin{align*}
\int \zeta \d \overline g_R = \int \zeta\left(R \frac z {\lvert z\rvert}\right)g^\prime(\d z).
\end{align*}

We now turn to establishing the claimed estimates for $\overline g_R$. Note that, directly from the definitions of $\tilde \pi$, $g^\prime$ and $\overline g_R$, an admissible plan for $W_c(g_R,\overline g_R)$ is
\begin{align*}
\int_{\Omega \cap \{X(\tau)\in \p B_R\}} \zeta\left(X(\tau),R\frac z {\lvert z\rvert}\right)\d \tilde \pi.
\end{align*}
Indeed, due to the definition of $\tilde \pi$ and the definition of $g_R$ in \eqref{eq:2}, for any test function $\zeta$,
\begin{align*}
\int_{\Omega \cap \{X(\tau)\in \p B_R\}} \zeta(X(\tau))\d\tilde\pi = \int_{\Omega \cap \{X(\tau)\in \p B_R\}} \zeta(X(\tau))\d\pi= \int \xi \d g_R.
\end{align*}
On the other hand, using \eqref{eq:1} and the definition of $\overline g_R$, 
\begin{align*}
\int_{\Omega \cap \{X(\tau)\in \p B_R\}} \zeta\left(R \frac z {\lvert z\rvert}\right)\d\tilde\pi = \int \zeta\left(R\frac z {\lvert z\rvert}\right) \d g^\prime = \int \zeta \d \overline g_R.
\end{align*}

In particular, using the $p$-growth of $c$ \eqref{ass:growth},
\begin{align*}
W_c(g_R,\overline g_R)\lesssim \int_{\Omega \cap \{X(\tau)\in \p B_R\}} c\left(X(\tau)-R\frac z {\lvert z\rvert}\right)\d\tilde \pi\lesssim \int_{\Omega \cap \{X(\tau)\in \p B_R\}} \left\lvert X(\tau)-R\frac z {\lvert z\rvert}\right\rvert^p\d\tilde \pi.
\end{align*}
Noting that $\lvert X(\tau)-R\frac z {\lvert z\rvert}\rvert \leq 2 \lvert X(\tau)-z\rvert$, we deduce
\begin{align*}
\left\lvert X(\tau)-R\frac z {\lvert z\rvert}\right\rvert^p \lesssim \lvert x-y\rvert^p+\lvert y-z\rvert^p.
\end{align*}
Thus, we deduce using once again the $p$-growth of $c$ \eqref{ass:growth} and Corollary \ref{cor:linfty},
\begin{align*}
W_c(g_R,\overline g_R) \lesssim \int_{\Omega \cap \{\exists t\in[0,1]\colon X(t)\in \p B_R\}} c(x-y)\d \pi+ D(4)\lesssim E(4)^{1+\frac 1 {p+d}}+D(4).
\end{align*}
Choosing $\e$ sufficiently small, the first estimate holds.

\noindent Noting $\sup g^\prime\leq \kappa_{\mu,4}\lesssim 1$, in order to prove the second inequality, it suffices to prove
\begin{align*}
\int_2^3 \int \lvert R-\lvert x\rvert\rvert^{p-1} \d g^\prime \lesssim E(4)+D(4)
\end{align*}
and to apply Lemma \ref{lem:projection}. The condition on the support of $g$ in Lemma \ref{lem:projection} applies due to Lemma \ref{lem:linfty}.
Note that by definition of $g^\prime$,
\begin{align*}
\int \lvert R-\lvert x\rvert\rvert^{p-1}\d g^\prime =& \int_{\Omega \cap \{X(\tau)\in \p B_R\}} \lvert \lvert z\rvert-R\rvert^{p-1}\tilde\pi(\d x\d y \d z)\\
\lesssim& \int_{\{y\in B_4\}\cap \{\min_{[0,1]}\lvert X\rvert\leq R\leq \max_{[0,1]} \lvert X\rvert\}} \lvert x-y\rvert^{p-1}+\lvert y-z\rvert^{p-1}\tilde \pi(\d x \dy \d z).
\end{align*}
In order to obtain the second line, we observed that since $\lvert X(\tau)\rvert = R$, it holds that ${\lvert \lvert z\rvert-R\rvert \leq \lvert x-y\rvert + \lvert y-z\rvert}$. In addition we noted that, since $X(\tau)\in \p B_R$, we have ${\min_{[0,1]}X\leq R\leq \max_{[0,1]} X}$ and $X(1)\in B_4$ due to Lemma \ref{lem:linfty}. Integrating over $R$, this gives
\begin{align*}
&\int_2^3 \int \lvert \lvert z\rvert -R\rvert^{p-1} \d g^\prime\d R \\
\lesssim& \int_{\{y\in B_4\}} (\max_{[0,1]} X-\min_{[0,1]} X)\left(\lvert x-y\rvert^{p-1}+\lvert y-z\rvert^{p-1}\right)\tilde \pi(\d x \d y \d z)\\
\leq& \int_{\{y\in B_4\}} \lvert x-y\rvert\left(\lvert x-y\rvert^{p-1}+\lvert y-z\rvert^{p-1}\right)\tilde \pi(\d x \d y \d z)\\
\lesssim& \int_{\{y\in B_4\}} \lvert x-y\rvert^p\pi(\d x \d y) + \int \lvert y-z\rvert^p \overline\pi(\d y \d z)\\
\leq & E(4)+D(4).
\end{align*}
The second-to last line was obtained applying Young's inequality. This concludes the proof.
\end{proof}

\section{Restricting the data}\label{sec:data}

As we see from Lemma \ref{lem:approximation}, we will not be able to work on $B_4$ directly, but will have to pass to a smaller ball $B_R$ with some suitably chosen $R\in[2,3]$. Hence, we need to control-- for a well-chosen $R$-- $D(R)$, while at the moment we only control $D(4)$. Unfortunately, this does not follow immediately from the definition but requires a technical proof utilising ideas of the previous sections. The outcome of these considerations is the following lemma:
\begin{lemma}\label{lem:dataRestriction}
For any non-negative measure $\mu$ there is $\e>0$ such that if $D(4)\leq \e$, then
\begin{align*}
\int_2^3 \left(W_{c}(\mu\llcorner B_R,\kappa_{\mu,R}\d x \llcorner B_R)+\frac 1 {\kappa_{\mu,R}}(\kappa_{\mu,R}-1)^p\right)\d R\lesssim D(4).
\end{align*}
\end{lemma}
\begin{proof}
Fix $R\in [2,3]$. In this proof $\pi$ will denote the optimal transference plan for the problem ${W_{c}(\mu\llcorner B_4,\kappa_{\mu,4}\d x\llcorner B_4)}$. 
Note that due to the $L^\infty$-bound Lemma \ref{lem:linfty}, if $X(0)\in B_R$, then $X(1)\in B_4$ and if $X(1)\in B_R$, then $X(0)\in B_4$. Define the measures $0\leq f^\prime\leq \kappa_{\mu,4}$ on $\overline B_R$ and $0\leq g^\prime\leq \kappa_{\mu,4}$ on $\overline B_4\setminus B_R$, which record where exiting and entering trajectories end up by asking that for all test functions $\zeta$,
\begin{align*}
\int \zeta \d f^\prime\coloneqq \int_{\Omega\cap \{X(0)\not\in B_R\}\cap \{X(1)\in B_R\}}\zeta(X(1))\d\pi \\
\int \zeta \d g^\prime\coloneqq \int_{\Omega\cap \{X(0)\in B_R\}\cap \{X(1)\not\in B_R\}}\zeta(X(1))\d\pi.
\end{align*}
Introduce the mass densities
\begin{align*}
\kappa_f = \frac{f^\prime(\R^d)} {\lvert B_R\rvert}\leq \kappa_{\mu,R}, \qquad \kappa_g = \frac{g^\prime(\R^d)}{\lvert B_R\rvert}.
\end{align*}
We use Lemma \ref{lem:triangleInequality} to deduce
\begin{align*}
&W_c(\mu \llcorner B_R,\kappa_{\mu,R}\d x \llcorner B_R)\\
\lesssim& W_c(\mu\llcorner B_R,\kappa_{\mu,4}\d x\llcorner B_R- f^\prime+g^\prime) + W_c(\kappa_{\mu,4}\d x\llcorner B_R-f^\prime + g^\prime,(\kappa_{\mu,4}-\kappa_f)\d x \llcorner B_R+g^\prime)\\
&\quad+ W_c((\kappa_{\mu,4}-\kappa_f)\d x \llcorner B_R+g^\prime,\kappa_{\mu,R}\d x \llcorner B_R)\\
=& I + II + III.
\end{align*}
Restricting $\pi$ to trajectories that start in $B_R$ gives an admissible plan for $I$. Consequently,
\begin{align*}
I=W_c(\mu\llcorner B_R,\kappa_{\mu,4}\d x\llcorner B_R-f^\prime + g^\prime)\leq W_c(\mu\llcorner B_4,\kappa_{\mu,4}\d x \llcorner B_4)\leq D(4).
\end{align*}

Since $II$ will be estimated in the same way as $III$, but is slightly more tricky, we first estimate $III$. In order to estimate $III$, introduce the projection $\hat g$ of $g^\prime$ onto $\p B_R$ via \eqref{eq:radial}.
Using Lemma \ref{lem:triangleInequality}, we deduce
\begin{align}\label{eq:data1}
III\lesssim& W_c((\kappa_{\mu,4}-\kappa_f)\d x\llcorner B_R+\hat g,\kappa_{\mu,R}\d x \llcorner B_R)+W_c(\hat g, g^\prime).
\end{align}
Regarding the first term, we claim that
\begin{align*}
W_c((\kappa_{\mu,4}-\kappa_f)\d x \llcorner B_R+\hat g,\kappa_{\mu,R}\d x \llcorner B_R)\lesssim \int (\hat g)^p.
\end{align*}
Indeed, an admissible density-flux pair $(\rho,j)$ for the Benamou-Brenier formulation \eqref{eq:eulerianFormulation} is given by
\begin{align*}
\begin{cases}
	\rho_t = (\kappa_{\mu,4}\d x\llcorner B_4-\kappa_f+t\kappa_g) +(1-t)\hat g g\\
	j_t = \nabla c^\ast(\D\phi) \d x \llcorner B_R,
\end{cases}
\end{align*}
where $\phi$ solves
\begin{align*}
\begin{cases}
-\Div \nabla c^\ast(\D\phi) = \kappa_g \quad&\text{ in } B_R\\
\nu \cdot \nabla c^\ast(\D\phi) = \hat g &\text{ on } \p B_R.
\end{cases}
\end{align*}
We find, writing $s=(\kappa_{\mu,4}-\kappa_f+t\kappa_g)$, for any $\zeta$ supported in $B_R$,
\begin{align}\label{data2}
\int \zeta \d j_t -\int c^\ast(\zeta)\d \rho_t =& \int_{B_R} \zeta\cdot \nabla c^\ast(\D\phi)-c^\ast(\zeta)s\d x\nonumber\\
\leq& \int_{B_R}\int_0^1 s c \left(\frac 1 s\nabla c^\ast(\D\phi)\right)\d t \d x
\end{align}
To obtain the second line, we used the Fenchel-Young inequality. Assuming $\lvert s-1\rvert \ll 1$ for now, using the $p$-growth of $c$ \eqref{ass:growth} it is straightforward to see that
\begin{align*}
\int_{B_R}\int_0^1 s c \left(\frac 1 s \nabla c^\ast(\D\phi)\right)\d t \d x \lesssim \int_{B_R} c\left(\nabla c^\ast(\D\phi)\right)\d x \lesssim \int_{\p B_R} (\hat g)^p.
\end{align*}
To obtain the last inequality, we used the energy inequality \eqref{eq:alternativeEnergy}. We now estimate $\int (\hat g)^p$ and $W_c(\hat g,g^\prime)$ arguing exactly as in Lemma \ref{lem:dataRestriction} in order to conclude
\begin{align*}
\int \hat g^p+W_c(\hat g,g^\prime)\lesssim D(4).
\end{align*}

Since $\kappa_{\mu,R} = \kappa_{4,R}-\kappa_f+ \kappa_g$ and $\lvert\kappa_{4,R}-1\rvert\lesssim D(4)$, to deduce that $\lvert s-1\rvert \ll 1$ and to conclude the estimate of $III$, it suffices to show $\kappa_f^p+\kappa_g^p\lesssim D(4)$. By symmetry it suffices to consider $\kappa_g$. Since $\hat g$ is supported on $\p B_R$, using Young's inequality, we find
\begin{align}\label{eq:data2}
\kappa_g^p = \frac {\hat g(\R^d)^p}{\lvert B_R\rvert^p}\leq \frac{\lvert \p B_R\rvert^{p-1}}{\lvert B_R\rvert^p}\int_{\p B_R}(\hat g)^p\lesssim D(4).
\end{align}
This concludes the estimate for $III$.

It remains to estimate $II$. Using the subadditivity of $W_c$, we have
\begin{align*}
&W_c(\kappa_{\mu,4}\d x \llcorner B_R-f^\prime + g^\prime,(\kappa_{\mu,4}-\kappa_f)\d x \llcorner B_R+g^\prime)\\
\leq& W_c(\kappa_{\mu,4}\d x \llcorner B_R-f^\prime,(\kappa_{\mu,4}-\kappa_f)\d x \llcorner B_R)+W_c(g^\prime,g^\prime)\\
=& W_c(\kappa_{\mu,4}\d x \llcorner B_R-f^\prime,(\kappa_{\mu,4}-\kappa_f)\d x\llcorner B_R ).
\end{align*}
We want to proceed exactly as we did in the estimate for $III$, the only delicate issue being that we do not have $\kappa_{\mu,4}\d x \llcorner B_R-f^\prime \geq c>0$. This is necessary in \eqref{data2}. However, this can be remedied by using Corollary \ref{cor:addConstant} to deduce
\begin{align*}
&W_c(\kappa_{\mu,4}\d x \llcorner B_R-f^\prime, (\kappa_{\mu,4}-\kappa_f)\d x \llcorner B_R)\\
\lesssim& W_c((2\kappa_{\mu,4}-\kappa_f)\d x \llcorner B_R-f^\prime, 2(\kappa_{\mu,4}-\kappa_f)\d x \llcorner B_R).
\end{align*}
Note that since $\kappa_{\mu,4}\d x \llcorner B_R-f^\prime \geq 0$ and using \eqref{eq:data2}, after choosing $\e$ sufficiently small,
\begin{align*}
(2\kappa_{\mu,4}-\kappa_f)\d x \llcorner B_R-f^\prime\geq \kappa_{\mu,4}-\kappa_f\geq \frac 1 2\kappa_{\mu,4}\geq c>0.
\end{align*}
In order to obtain the last inequality, we used the definition of $D$ and chose $\e$ sufficiently small. This allows to proceed using the same argument as for $III$ to conclude
\begin{align*}
W_c(\kappa_{\mu,4}\d x \llcorner B_R-f^\prime, (\kappa_{\mu,4}-\kappa_f)\d x \llcorner B_R)\lesssim D(4).
\end{align*}
This completes the proof.
\end{proof}

\section{The \texorpdfstring{$c^\ast$}{}-harmonic approximation result}\label{sec:proof}
The goal of this section is to prove the $c^\ast$-harmonic approximation result. With the results of the previous sections in hand, we can give a more precise version of Theorem \ref{thm:main}. In particular, we can make explicit the problem that $\phi$ solves.

To this end, in light of Lemma \ref{lem:approximation} and Lemma \ref{lem:dataRestriction}, given $\tau>0$, we fix $R\in(2,3)$ such that there exist non-negative $\overline f_R$, $\overline g_R$ such that
\begin{align*}
W_c(f_R,\overline f_R)+W_c(g_R,\overline g_R)\lesssim \tau E(4)+D(4)\nonumber\\
D(R)+\int_{\p B_R} \overline f_R^p+\overline g_R^p \lesssim E(4)+D(4).
\end{align*}
Then let $\phi$ be a solution with $\int_{B_R}\phi\d x = 0$ of
\begin{align}\label{eq:phi}
\begin{cases}
-\Div \nabla c^\ast(\D\phi) = c_R \quad&\text{ in } B_R\\
\nabla c^\ast(\D\phi)\cdot\nu = \overline g_R-\overline f_R &\text{ on } \p B_R.
\end{cases}
\end{align}
where $c_R = \lvert B_R\rvert^{-1}\left(\int_{\p B_R} \overline g_R-\overline f_R\right)$ is the constant so that \eqref{eq:phi} is well-posed.
We emphasize that while we do not make the dependence explicit in our notation, $\phi$ depends on the choice of radius $R$.

With this notation in place, we state a precise version of our main result.
\begin{theorem}\label{thm:mainBody}
For every $\tau>0$, there exist positive constants $\e(\tau),C(\tau)>0$ such that if $E(4)+D(4)\leq \e(\tau)$, then there exists $R\in(2,3)$ such that
\begin{align*}
\int_{\#_1} c(x-y-\nabla c^\ast(\D\phi)) \leq \tau E(4) + C(\tau) D(4),
\end{align*}
where $\phi$ solves \eqref{eq:phi}.
\end{theorem}

The proof will be a direct consequence of the lemmata we prove in the following subsections. We begin in Section \ref{sec:quasiOrthog} by using strong $p$-convexity in order to bound a quantity related to the left-hand side of the estimate in Theorem \ref{thm:mainBody} by a difference of energies, as well as two error terms. The first error term arises from the approximation of the boundary data, while the second error term comes from passing to the perspective of trajectories. We construct a competitor to estimate the difference of energies and estimate the two error terms in Section \ref{sec:errorEstimates}. Collecting estimates, we conclude the proof of Theorem \ref{thm:mainBody} in Section \ref{sec:conclusion}.

\subsection{Quasi-orthogonality}\label{sec:quasiOrthog}
The key observation in order to prove Theorem \ref{thm:mainBody} is contained in the following elementary lemma, which relies on the strong $p$-convexity of $c$.
\begin{lemma}\label{lem:orthog}
For any $\pi\in \Pi(\lambda,\mu)$ and $\phi$ continuously differentiable in $\overline B_R$, there is $c(p,\Lambda)$ such that we have
\begin{align*}
&c(p,\Lambda)\int_\Omega \int_\sigma^\tau V\left(\dot X(t),\nabla c^\ast(\nabla\phi(X(t))\right)\d t\d\pi \\
\leq& \int_\Omega c(x-y)\d\pi - \int_{B_R}c(\nabla ^\ast c(\D\phi))\d x - \int_\Omega \int_\sigma^\tau \langle \dot X(t)-\nabla c^\ast(\D \phi(X(t)),\D\phi(X(t))\rangle \d t\d\pi \\
&\quad+ \int_{B_R}c(\nabla c^\ast(\D\phi))\d x-\int_{\Omega}\int_\sigma^\tau c(\nabla c^\ast(\D\phi(X(t)))\d t\d\pi.
\end{align*}
\end{lemma}
\begin{proof}
We apply \eqref{ass:smoothness} with $x = \dot X(t)$ and $y=\nabla c^\ast(\D\phi(X(t)))$. Noting that we have ${\nabla c(\nabla c^\ast(\D\phi)) = \D\phi}$ and $\int_\Omega \int_\sigma^\tau c(\dot X(t))\d t\d\pi \leq \int_\Omega c(x-y)\d\pi$ this gives the desired result.
\end{proof}
\subsection{Error estimates}\label{sec:errorEstimates}
We would like to apply Lemma \ref{lem:orthog} with $\phi$ solving \eqref{eq:phi}. In order to do so, we require $\phi\in C^1(\overline B_R)$. However note that $\overline g_R$ and $\overline f_R$ will in general not be sufficiently smooth to ensure that $\phi\in C^1(\overline B_R)$. Thus, we approximate them using mollification. To be precise, let $0<r\ll 1$ and denote by $\overline f_R^r$ and $\overline g_R^r$, respectively, the convolution with a smooth convolution kernel (on $\p B_R)$ at scale $r$ of $\overline f_R$ and $\overline g_R$. Set $\phi^r$ to be the solution with $\int_{B_R} \phi^r\d x = 0$ of
\begin{align}\label{eq:phiMollif}
\begin{cases}
-\Div \nabla c^\ast(\D\phi^r) = c^r \quad&\text{ in } B_R\\
\nabla c^\ast(\D\phi^r)\cdot\nu = \overline g_R^r-\overline f_R^r &\text{ on } \p B_R.
\end{cases}
\end{align}
Here $c^r= \lvert B_R\rvert^{-1}\left(\int_{\p B_R} \overline g_R^r-\overline f_R^r\right)$ is the constant such that \eqref{eq:phiMollif} is well-posed.

We begin by showing that replacing $\overline f_R$ and $\overline g_R$ with $\overline f_R^r$ and $\overline g_R^r,$ respectively is not detrimental on the left-hand side of the estimate in Lemma \ref{lem:orthog}.
\begin{lemma}\label{lem:lefthandSide}
For every $0<\tau$ there exists $\e(\tau)$ and $C(\tau), r_0(\tau)>0$, such that if it holds that ${E(4)+D(4)\leq \e(\tau)}$ and $0<r\leq r_0$, then there exists $R\in[2,3]$ such that if $\phi$ solves \eqref{eq:phi} and $\phi^r$ solves \eqref{eq:phiMollif}, then
\begin{align*}
&\int_{\Omega_{3/2}} \int_{\sigma_{3/2}}^{\tau_{3/2}} V\left(\dot X(t),\nabla c^\ast(\D\phi(X(t))\right)\d t\d\pi-\int_{\Omega_{3/2}} \int_{\sigma_{3/2}}^{\tau_{3/2}} V\left(\dot X(t),\nabla c^\ast(\D\phi^r(X(t))\right)\d t\d\pi\\
\lesssim& \tau (E(4)+D(4)).
\end{align*}
\end{lemma}
\begin{proof}
Write $\xi(x) = \nabla c^\ast(\D\phi(x))$, $\xi^r(x) = \nabla c^\ast(\D\phi^r(x))$. We focus on the case $p\leq 2$. The case $p>2$ follows by similar arguments, but is easier.
In light of \eqref{eq:VDiff}, the controlled $p^\prime$-growth of $c^\ast$ \eqref{eq:c1growthDual} and using H\"older, we find
\begin{align*}
&\left\lvert\int_{\Omega_{3/2}}\int_{\sigma_{3/2}}^{\tau_{3/2}} V\left(\dot X(t),\xi(X(t))\right)\d t\d\pi-\int_\Omega \int_\sigma^\tau V\left(\dot X(t),\xi^r(X(t))\right)\d t\d\pi\right\rvert\\
\lesssim& \int_{\Omega_{3/2}}\int_{\sigma_{3/2}}^{\tau_{3/2}} \lvert \xi(X(t))-\xi^r(X(t))\rvert\left(\lvert \dot X(t)+\lvert \xi(X(t))\rvert + \lvert \xi^r(X(t))\rvert\right)^{p-1}\d t\d\pi\\
\lesssim& \int_{\Omega_{3/2}}\int_{\sigma_{3/2}}^{\tau_{3/2}} \lvert \D\phi(X(t))-\D\phi^r(X(t))\rvert \left(\lvert \D\phi(X(t))\rvert + \lvert \D\phi^r( X(t))\rvert\right)^{p^\prime-1}\d t\d\pi\\
\leq& \left(\int_{\Omega_{3/2}}\int_{\sigma_{3/2}}^{\tau_{3/2}}\lvert \D\phi(X(t))-\D\phi^r(X(t))\rvert^{p^\prime}\d t\d\pi\right)^\frac 1 {p^\prime}\\
&\quad \times\left(\int_{\Omega_{3/2}}\int_{\sigma_{3/2}}^{\tau_{3/2}}\lvert \D\phi (X(t))\rvert^{p^\prime}+\lvert \D\phi^r (X(t))\rvert^{p^\prime}\d t\d\pi\right)^\frac 1 p.
\end{align*}
Note that due to Lemma \ref{lem:linfty}, if $X\in\Omega_{3/2}$, then $X(0)\in B_{7/8}$. Thus using elliptic regularity in the form of \eqref{eq:interior} and \eqref{eq:regularityPhiR},
\begin{align*}
&\Big\lvert\int_{\Omega_{3/2}}\int_{\sigma_{3/2}}^{\tau_{3/2}}\lvert \D\phi(X(t))-\D\phi^r(X(t))\rvert^{p^\prime}\d t\d\pi-\int_{\Omega_{3/2}}\int_{\sigma_{3/2}}^{\tau_{3/2}}\lvert \D\phi(X(0))-\D\phi^r(X(0))\rvert^{p^\prime}\d t\d\pi\Big\rvert\\
\lesssim& \left([\D\phi]_{C^{0,\beta}(B_{7/8})}+[\D\phi^r]_{C^{0,\beta}(B_{7/8})}\right)\left(\sup_{B_{7/8}}\lvert \D \phi\rvert+\lvert \D \phi^r\rvert\right)^{p^\prime-1}\int_{\Omega_{3/2}}\int_{\sigma_{3/2}}^{\tau_{3/2}} \lvert X(t)-X(0)\rvert^\beta \d t\d\pi\\
\lesssim& c(r)(E(4)+D(4)) \pi(\Omega_{3/2})(E(4)+D(4))^\frac \beta {p+d}\lesssim c(r)(E(4)+D(4))^{1+\frac \beta {p+d}}.
\end{align*}
In order to obtain the last line, we used the $L^\infty$-bound Lemma \ref{lem:linfty}. Moreover, using Lemma \ref{lem:C2measures} and elliptic regularity in the form of \eqref{eq:diff},
\begin{align*}
&\int_{\Omega_{3/2}}\int_{\sigma_{3/2}}^{\tau_{3/2}} \lvert \D \phi(X(0))-\D \phi^r(X(0))\rvert^{p^\prime}\d t\d\pi\leq \int_{B_{7/8}} \lvert \D \phi(x)-\D \phi^r(x)\rvert^{p^\prime}\d\mu\\
\leq& \Big\lvert \int_{B_{7/8}} \lvert \D \phi(x)-\D \phi^r(x)\rvert^{p^\prime}(\d\mu-\kappa_{\mu,R}\d x)\Big\rvert+\kappa_{\mu,R}\int_{B_R}\lvert \D \phi(x)-\D \phi^r(x)\rvert^{p^\prime}\d x\\
\lesssim& \left(\sup_{B_{7/8}} \lvert \D \phi\rvert^{p^\prime-1}+\lvert \D\phi^r\rvert^{p^\prime-1}\right)\left([\D \phi]_{C^{0,\beta}(B_{(3/2+R)/2})}+[\D^r\phi]_{C^{0,\beta}(B_{(3/2+R)/2})}\right)\\
&\quad \times W_{c}(\mu\llcorner B_R,\kappa_{\mu,R}\d x\llcorner B_R)^\frac \beta p+ r^s(E(4)+D(4))\\
\lesssim& c(r)(E(4)+D(4))^{1+\frac \beta p} + r^s(E(4)+D(4)).
\end{align*}
Arguing similarly, that is first replacing $X(t)$ with $X(0)$, at the cost of making an error of size $c(r)\left(E(4)+D(4)\right)^{1+\frac \beta {p+d}}$, and then replacing $\d \mu$ with $\d x$, making an error $c(r)(E(4)+D(4))^{1+\frac \beta p}+r(E(4)+D(4))$, we find
\begin{align*}
&\int_{\Omega_{3/2}}\int_{\sigma_{3/2}}^{\tau_{3/2}}\lvert \nabla \phi(X(t))\rvert^{p^\prime}+\lvert \nabla \phi^r(X(t))\rvert^{p^\prime}\d t\d\pi\\
\lesssim& r^s(E(4)+D(4))+ c(r)(E(4)+D(4))^{1+\frac \beta {p+d}}.
\end{align*}
Collecting estimates and choosing $r$ as well as $\e(\tau)$ sufficiently small, the desired result follows.
\end{proof}

We now turn to estimating each of the three terms on the right-hand side of the estimate in Lemma \ref{lem:orthog} in turn. We will see that the second and third term are errors that arise from the approximation of the boundary data and from passing to the perspective of trajectories, respectively. Accordingly, estimating them will be essentially routine. In contrast, estimating the first term requires us to contrast an appropriate competitor to $\pi$.
\begin{lemma}\label{lem:firstTerm}
For every $0<\tau$ there exists $\e(\tau),C(\tau),r_0(\tau)>0$ such that if it holds that ${E(4)+D(4)\leq \e(\tau)}$ and $0<r\leq r_0$, then there exists $R\in[2,3]$ such that if $\phi^r$ solves \eqref{eq:phiMollif}, then
\begin{align*}
&\int_\Omega c(x-y)\d\pi - \int_{B_R} c(\nabla c^\ast(\D\phi^r))\d x \lesssim \tau E(4)+D(4).
\end{align*}
\end{lemma}
\begin{proof}
We note, in the case $p\leq 2$, using the $p$-growth of $c$ \eqref{ass:Cgrowth}, the $p^\prime-1$-growth of $\nabla c^\ast$ \eqref{eq:c1growthDual} and H\"older,
\begin{align*}
&\int_{B_R} c(\nabla c^\ast(\D\phi^r))-c(\nabla c^\ast(\D\phi))\d x\\
\lesssim& \int_{B_R} \lvert \nabla c^\ast(\D\phi^r)-\nabla c^\ast(\D\phi)\rvert(\lvert \nabla c^\ast(\D\phi^r)\rvert+\lvert \nabla c^\ast(\D\phi)\rvert)^{p-1}\d x\\
\lesssim& \int_{B_R} \lvert \D\phi-\D\phi^r\rvert \left(\lvert D\phi\rvert+\lvert \D\phi^\prime\rvert\right)^{p^\prime-1}\\
\lesssim& \|\D\phi-\D\phi^r\|_{L^{p^\prime}(B_R)}\left(\|\D\phi\|_{L^{p^\prime}(B_R)}+\|\D\phi^r\|_{L^{p^\prime}(B_R)}\right)^{p^\prime-1}\\
\lesssim& r^\frac s {p^\prime} (E(4)+D(4)).
\end{align*}
To obtain the last line, we used the elliptic estimates \eqref{eq:energy} and \eqref{eq:diff}.  In case $p\geq 2$ a similar estimate holds by the same argument. 
Due to the localisation result of Lemma \ref{lem:localisation} and the $L^\infty$-bound in the form of Corollary \ref{cor:linfty},
\begin{align*}
\int_\Omega c(x-y))\d\pi\leq W_c(\lambda\llcorner B_R+f_R,\mu\llcorner B_R+g_R)+2\int_{\Omega \cap \{\exists t\in[0,1]\colon X(t)\in \p B_R \} }c(x-y)\d\pi\\
\leq W_c(\lambda\llcorner B_R+f_R,\mu\llcorner B_R+g_R)+c\left(\tau E(4)+D(4)\right).
\end{align*}
In particular, combining the previous two estimates and choosing $r$ sufficiently small, it suffices to prove
\begin{align*}
&W_c(\lambda \llcorner B_R +f_R,\mu\llcorner B_R+g_R)- \int_{B_R} c(\nabla c^\ast(\D\phi))\d x \lesssim \tau E(4)+D(4).
\end{align*}
Using Lemma \ref{lem:triangleInequality}, we obtain for $\delta\in(0,1)$ to be fixed
\begin{align*}
&W_c(\lambda \llcorner B_R+f_R,\mu\llcorner B_R+g_R)\\
\leq& (1+\delta) W_c(\kappa_{\lambda,R}\d x\llcorner B_R+\overline f_R,\kappa_{\mu,R}\d x\llcorner B_R+\overline g_R)\\
&\quad+ c(\delta) \left(W_c(\lambda\llcorner B_R,\kappa_{\lambda,R} \d x\llcorner B_R)+W_c(\mu\llcorner B_R,\kappa_{\mu,R} \d x\llcorner B_R)\right)\\
&\qquad +c(\delta)\left(W_c(f_R,\overline f_R)+W_c(g_R,\overline g_R)\right).
\end{align*}
Noting that due to the definition of $D$ and our choice of $R$,
\begin{align*}
&W_c(\lambda\llcorner B_R,\kappa_{\lambda,R} \d x\llcorner B_R)+W_c(\mu\llcorner B_R,\kappa_{\mu,R} \d x\llcorner B_R)+W_c(f_R,\overline f_R)+W_c(g_R,\overline g_R)\\
\lesssim& \tau E(4)+D(4),
\end{align*}
we claim that for some $C>0$,
\begin{align}\label{claim:benamou}
W_c(\kappa_{\lambda,R} \d x\llcorner B_R+\overline f_R,\kappa_{\mu,R} \d x\llcorner B_R+\overline g_R)\leq (1+CD(4)^\frac 1 p) \int_{B_R} c(\nabla c^\ast(\D\phi))\d x.
\end{align}
Collecting estimates, choosing first $\delta$ and $r$ small, then $\e$ small, once \eqref{claim:benamou} is established, the proof is complete.

Establishing \eqref{claim:benamou} is easy to do using the Benamou-Brenier formulation \eqref{eq:eulerianFormulation}. For $t\in[0,1]$ introduce the non-singular, non-negative measure
\begin{align*}
\rho_t = t(\kappa_{\mu,R} \d x \llcorner B_R+\overline f_R)+(1-t)(\kappa_{\lambda,R} \d x\llcorner B_R+\overline g_R),
\end{align*}
and the vector-valued measure
\begin{align*}
j_t = \nabla c^\ast(\D\phi)\d x\llcorner B_R.
\end{align*}
Note that \eqref{eq:phiMollif} can be rewritten as
\begin{align*}
\frac{\d}{\d t} \int \zeta \d\rho_t = \int \nabla \zeta \cdot \d j_t
\end{align*}
for all test functions $\zeta$, that is $(j_t,\rho_t)$ satisfy the continuity equation in the sense of distributions.
The Benamou-Brenier formula \eqref{eq:eulerianFormulation} gives
\begin{align*}
W_c(\kappa_{\lambda,R} \d x \llcorner B_R+\overline f_R,\kappa_{\mu,R}\d x\llcorner B_R+\overline g_R)\leq \int_0^1 \int c\left(\frac{\d j_t}{\d \rho_t}\right)\d \rho_t \d t.
\end{align*}
Here the right-hand side needs to be interpreted in the sense of \eqref{eq:BenamouBrenierDefinition}. Since $j_t$ is supported in $B_R$ it suffices to consider $\zeta$ supported in $B_R$. Then by definition of $(j_t,\rho_t)$ and the Fenchel-Young inequality for any $s>0$,
\begin{align*}
\int \zeta \d j_t-\int c^\ast(\zeta)\d\rho_t =& \int_{B_R}\zeta \cdot \nabla c^\ast(\D\phi)- c^\ast(\zeta) (t\kappa_{\mu,R}+(1-t)\kappa_{\lambda,R})\d x\\
\leq& \int_{B_R}s c^\ast(\zeta)+s c\left(\frac 1 s \nabla c^\ast(\D\phi)\right)-c^\ast(\zeta)(t\kappa_{\mu,R}+(1-t)\kappa_{\lambda,R})\d x.
\end{align*}
Choosing $s = t\kappa_{\mu,R}+(1-t)\kappa_{\lambda,R}$ and integrating in $t$, we deduce
\begin{align*}
W_c(\kappa_{\mu,R}\d x\llcorner B_R+\overline f_R,\kappa_{\lambda,R} \d x\llcorner B_R+\overline g_R)\leq \int_{B_R}\int_0^1 s c\left(\frac{\nabla c^\ast(\D\phi))}{s}\right)\d t\d x.
\end{align*}
Now
\begin{align*}
\int_{B_R}\int_0^1 s c\left(\frac{\nabla c^\ast(\D\phi))}{s}\right)\d t\d x\leq& \int_{B_R} c(\nabla c^\ast(\D\phi))\d x + \int_{B_R}\int_0^1 (s-1) c(\nabla c^\ast(\D\phi))\ d x\d t\\
&\quad + \int_{B_R}\int_0^1 s \left(c\left(\frac{\nabla c^\ast(\D\phi)} s\right)-c(\nabla c^\ast(\D\phi))\right)\d t \d x.
\end{align*}
Note that 
\begin{align*}
\lvert s-1\rvert \lesssim D(4)^\frac 1 p.
\end{align*}
Further using \eqref{ass:Cgrowth} and \eqref{ass:growth},
\begin{align*}
&\int_{B_R}\int_0^1 s \left(c\left(\frac{\nabla c^\ast(\D\phi)} s\right)-c(\nabla c^\ast(\D\phi))\right)\d t \d x\\
\lesssim& \int_{B_R}\int_0^1 \left\lvert 1- \frac 1 s\right\rvert \left(1+\frac 1 s\right)^{p-1} \lvert \nabla c^\ast(\D\phi)\rvert^{p} \d t\d x
\lesssim D(4)^\frac 1 p \int_{B_R} c(\nabla c^\ast(\D\phi))\d x.
\end{align*}
Thus the proof of \eqref{claim:benamou} is complete.
\end{proof}

We turn to the second term on the right-hand side of Lemma \ref{lem:orthog}. This term will be small due to the definition of $\phi^r$.
\begin{lemma}\label{lem:secondTerm}
For every $0<\tau$ there exists $\e(\tau),C(\tau),r_0(\tau)>0$ such that if it holds that ${E(4)+D(4)\leq \e(\tau)}$ and $0<r\leq r_0$, then there exists $R\in[2,3]$ such that if $\phi^r$ solves \eqref{eq:phiMollif}, then
\begin{align*}
&\int_\Omega \int_{\sigma}^\tau \langle \dot X(t)-\nabla c^\ast(\D\phi^r(X(t)),\D\phi^r(X(t))\rangle\d t\d\pi \lesssim \tau E(4)+D(4).
\end{align*}
\end{lemma}
\begin{proof}
Note that $\frac{\d}{\d t} \phi^r(X(t)) = \langle \dot X(t),\nabla \phi^r(X(t))\rangle$. Thus, since $\pi\in \Pi(\lambda,\mu)$ as well as using the definition of $f_R$ and $g_R$,
\begin{align*}
\int_\Omega \int_\sigma^\tau \langle \dot X(t),\D\phi^r(X(t))\rangle\d t\d\pi =& \int_\Omega \phi^r(X(\tau)-\phi^r(X(\sigma))\d\pi\\
=& \int_{\Omega\cap\{X(\tau)\in \p B_R\}} \phi^r(X(\tau))\d \pi+\int_{\{x\in B_R\}} \phi^r(x)\d \pi\\
&\quad-\int_{\Omega \cap\{X(\sigma)\in \p B_R\}}\phi^r(X(\sigma))\d\pi - \int_{\{y\in B_R\}}\phi^r(y)\d\pi\\
=& \int_{B_R}\phi^r \d(\mu-\lambda) + \int_{\p B_R} \phi^r \d(g_R-f_R).
\end{align*}
On the other hand, as in Lemma \ref{lem:lefthandSide}, at cost of an error $c(r)(E(4)+D(4))^{1+\frac \beta {p+d}}$, we may replace $\D\phi^r(X(t))$ with $\D\phi^r(x)$ in the expression
\begin{align*}
\int_\Omega \int_\sigma^\tau \langle \nabla c^\ast(\D\phi^r(X(t)),\D\phi^r(X(t))\rangle \d t\d\pi
\end{align*}
and $\int_\Omega \int_\sigma^\tau\d t\d\pi$ with $\int_{B_R}\d x$ at the cost of a further error $c(r)(E(4)+D(4))^{1+\frac \beta p}$. Thus it suffices to consider
\begin{align*}
\int_{B_R} \langle \nabla c^\ast(\D\phi),\D\phi\rangle \d x = c^r\int_{B_R} \phi^r \d x+\int_{\p B_R} \phi^r\d\left(\overline g^r_R-\overline f^r_R\right).
\end{align*}
Collecting estimates, we have shown
\begin{align*}
&\int_\Omega \int_{\sigma}^\tau \langle \dot X(t)-\nabla c^\ast(\D\phi^r(X(t)),\D\phi^r(X(t))\rangle\d t\d\pi \\
\lesssim& \int_{B_R} \phi^r \d(\mu-\lambda-c^r)+\int_{\p B_R} \phi^r \d(g_R-\overline g_R^r+f_R-\overline f_R^r)+c(r) \left(E(4)+D(4)\right)^{1+\frac \beta {p+d}} \\
=& I + II + III.
\end{align*}
We find using Lemma \ref{lem:C2measures}, Young's inequality and elliptic regularity in the form of \eqref{eq:regularityPhiR} as well as \eqref{eq:energy},
\begin{align*}
I\leq& \lvert \int_{B_R} \phi^r \d(\mu-\kappa_{\mu,R}-\lambda+\kappa_{\lambda,R})\rvert + \lvert \int_{B_R} \phi^r \d(\kappa_{\mu,R}-\kappa_{\lambda,R}-c^r)\rvert\\
\lesssim& \sup_{B_R}\lvert \D\phi^r\rvert \left(W_{c}(\lambda\llcorner B_R,\lambda_{\mu,R}\d x\llcorner B_R)^\frac 1 p+W_{c}(\mu\llcorner B_R,\kappa_{\mu,R}\d x\llcorner B_R)^\frac 1 p\right)\\
&\quad+\rvert \kappa_{\mu,R}-\kappa_{\lambda}-c^r\rvert \|\phi^r\|_{\LL^p(B_R)} \\
\lesssim& \tau E(4)+ c(\tau,r) D(4)+c(\tau)\lvert \kappa_{\mu,R}-\kappa_{\lambda,R}-c^r\rvert^{p^\prime}.
\end{align*}
Noting that as $r\to 0$, 
$$c^r=\lvert B_R\rvert^{-1} \left(g^r(\p B_R)-f^r(\p B_R)\right)\to \lvert B_R\rvert^{-1}\left(g(\p B_r)-f(\p B_R)\right) = \kappa_{\mu,R}-\kappa_{\lambda,R},
$$
 we deduce, choosing $r_0$ sufficiently small, that
\begin{align*}
I =\int_{B_R} \phi^r \d(\mu-\lambda-c^r)\lesssim \tau E(4)+D(4).
\end{align*}
Finally consider $II$. By symmetry it suffices to consider the terms involving $g$. We estimate, denoting by $(\phi^r)^r$ convolution of $\phi^r$ with the convolution kernel used to construct $g_R^r$,
\begin{align*}
\int_{\p B_R} \phi^r \d(g_R-\overline g_R^r) = \int_{\p B_R} (\phi^r)^r-\phi^r\d\overline g_R+\int_{\p B_R}\phi^r \d(\overline g_R-g_R).
\end{align*}
Now note that a standard mollification argument shows
\begin{align*}
\int_{\p B_R}\phi^r \d(g_R-\overline g_R^r)\lesssim r^\frac 1 p\|\D\phi^r\|_{\LL^{p^\prime}(B_R)}\|\overline g_R\|_{\LL^p(\p B_R)}\lesssim r^\frac 1 p (E(4)+D(4)),
\end{align*}
and that moreover using Lemma \ref{lem:C2measures} and Young's inequality,
\begin{align*}
\int_{\p B_R} \phi^r(\d\overline g_R-g)\lesssim [\D_{\mathrm{tan}}\phi^r]_{C^{0,\beta}(\p B_R)} W{c}(\overline g_R,\overline g)^\frac 1 p\lesssim c(r)(\tau E(4)+D(4)).
\end{align*}
Thus, collecting estimates and first choosing $r_0$ sufficiently small, then $\e$ small, we conclude the proof.
\end{proof}

We next estimate the third term on the right-hand side of the estimate in Lemma \ref{lem:orthog}.
\begin{lemma}\label{lem:thirdTerm}
For every $0<\tau$ there exists $\e(\tau),C(\tau),r_0(\tau)>0$ such that if it holds that ${E(4)+D(4)\leq \e(\tau)}$ and $0<r\leq r_0$, then there exists $R\in[2,3]$ such that if $\phi^r$ solves \eqref{eq:phiMollif}, then
\begin{align*}
&\int_{B_R}c(\nabla c^\ast(\D\phi^r))\d x-\int_\Omega \int_{\sigma}^\tau c(\nabla c^\ast(\D\phi^r(X(t)))\d t\d\pi \lesssim \tau E(4)+D(4).
\end{align*}
\end{lemma}
\begin{proof}
Set $\xi = c(\nabla c^\ast(\D\phi^r))$. Then
\begin{align*}
&\int_{B_R} c(\nabla c^\ast(\D\phi^r))\d x-\int_\Omega \int_\sigma^\tau c(\nabla c^\ast(\D\phi^r(X(t))\d t\d\pi\\
=&(1-\kappa_{\mu,R})\int_{B_R}\xi\d x+\left(\kappa_{\mu,R} \int_{B_R}\xi\d x-\int_{B_R\times \R^d} \xi \d\pi\right)\\
&\quad+\int_{B_R\times \R^d}\xi\d\pi-\int_\Omega\int_\sigma^\tau \xi(X(t))\d t\d\pi\\
=& I + II + III.
\end{align*}
Using \eqref{eq:alternativeEnergy} and Young's inequality, we find
\begin{align*}
I \leq D(4)^\frac 1 p (E(4)+D(4))\lesssim \tau E(4) + D(4).
\end{align*}
Employing Lemma \ref{lem:C2measures} we deduce
\begin{align*}
II\lesssim \|\xi\|_{C^{0,\beta}(\overline B_R)}W_{c}(\mu\llcorner B_R,\kappa_{\mu,R}\llcorner B_R)^\frac \beta p.
\end{align*}
It is straightforward to check that $\|\xi\|_{C^{0,\beta}(\overline B_R)}\lesssim \|\D\phi^r\|_{C^{0,\beta}(\overline B_R)}\left(\sup_{\overline B_R}\D\phi^r\right)^{p^\prime-1}$, so that using \eqref{eq:regularityPhiR} and Young's inequality,
\begin{align*}
II\lesssim c(r)(E(4)+D(4))D(4)^\frac \beta p\lesssim \tau E(4)+D(4).
\end{align*}
In order to estimate $III$, we first find
\begin{align*}
&I(X(0)\in B_R)\xi(X(0))-I(X(t)\in B_R)\xi(X(t))\\
\leq& I(\exists s\in [0,1]\colon X(s)\in \p B_R, X(0)\in \overline B_R)\xi(X(0))\\
&\quad+ I(\forall s\in[0,1]\; X(s)\in B_R)\left(\xi(X(0))-\xi(X(t)\right).
\end{align*}
Thus, using also \eqref{eq:regularityPhiR} and Jensen's inequality,
\begin{align*}
III\leq&\int_0^1\int I(X(0)\in B_R)\xi(X(0))-I(X(t)\in B_R)\xi(X(t))\d \pi\d t\\
\leq& \sup_{\overline B_R}\lvert \xi\rvert \int_0^1\int I(\exists s\in[0,1]\colon X(s)\in \p B_R, X(0)\in \p B_R)\d\pi\d t\\
&\quad + \|D\xi\|_{C^{0,\beta}(\overline B_R)}\int_0^1 \int I(\forall s\in[0,1]\; X(s)\in B_R)\lvert X(t)-X(0)\rvert^\beta\d t\d\pi\\
\lesssim& c(r)(E(4)+D(4))\pi(\exists t\in[0,1]\colon X(t)\in \p B_R)+c(r)(E+D)\int_\Omega \lvert x-y\rvert^\beta \d\pi\\
\lesssim& c(r) (E(4)+D(4))^{1+\frac 1 {p+d}}+c(r)(E(4)+D(4))^{1+\frac \beta p}.
\end{align*}
In order to obtain the last line we used Corollary \ref{cor:linfty}. Collecting estimates and choosing first $r_0$ small, then $\e$ small, we conclude the proof.
\end{proof}

\subsection{Proof of Theorem \ref{thm:mainBody}}\label{sec:conclusion}
We are now ready to prove Theorem \ref{thm:mainBody}.
\begin{proof}[Proof of Theorem \ref{thm:mainBody}]
Applying Lemma \ref{lem:orthog} to $\phi^r$ and collecting the output of Lemma \ref{lem:lefthandSide}, Lemma \ref{lem:firstTerm}, Lemma \ref{lem:secondTerm} and Lemma \ref{lem:thirdTerm}, we have shown that for any $0<\tau$, there is $\e,C>0$ such that if $E(4)+D(4)\leq \e$, then
\begin{align*}
\int_{\Omega_{3/2}} \int_{\sigma_{3/2}}^{\tau_{3/2}}V\left(\dot X(t),\nabla c^\ast(\D \phi(X(t))\right)\d t\d\pi \leq \tau E(4)+C D(4).
\end{align*}
Arguing as in Lemma \ref{lem:lefthandSide}, we may replace $\D(\phi(X(t)))$ by $\D(\phi(X(0)))$ at the cost of an error of size $\left(E(4)+D(4)\right)^{1+\frac \beta {p+d}}$. Noting that due to Lemma \ref{lem:linfty}, $\#_1\subset \Omega_{3/2}$, we find
\begin{align*}
\int_{\#_1}V\left(x-y-\nabla c^\ast(\D\phi(x))\right)\d \pi\leq \int_{\Omega_{3/2}}\int_0^1 V\left(\dot X(t)-\nabla c^\ast(\D\phi(X(0))\right)\d t\d\pi.
\end{align*}
Now using \eqref{eq:CgrowthDual}, elliptic regularity \eqref{eq:interior}, as well as the $L^\infty$-bound in the form of Corollary \ref{cor:linfty},
\begin{align*}
&\int_{\Omega_{3/2}\cap (\{\sigma_{3/2}>0\} \cup \{\tau_{3/2}<1\})}\int_{\sigma_{3/2}}^{\tau_{3/2}} V\left(\dot X(t)-\nabla c^\ast(\D\phi(X(0))\right)\d t\d\pi \\
\lesssim& \int_{\Omega_{3/2}\cap (\{\sigma_{3/2}>0\} \cup \{\tau_{3/2}<1\})} \lvert \dot X(t)\rvert^p+\lvert \D\phi(X(0)))\rvert^{p^\prime}\\
\lesssim& (E(4)+D(4)) \pi(\Omega_{3/2}\cap (\{\sigma_{3/2}>0\} \cup \{\tau_{3/2}<1\}))\\
\lesssim& (E(4)+D(4))^{1+\frac 1 {p+d}}.
\end{align*}
Thus, we conclude
\begin{align*}
\int_{\#_1}V(x-y-\nabla c^\ast(\D\phi(x)))\d \pi\lesssim \tau E(4)+D(4).
\end{align*}
This concludes the proof in the case $p\geq 2$. In the case $p\leq 2$, an application of H\"older's inequality combined with \eqref{ass:growth} concludes the proof.
\end{proof}
\bibliographystyle{acm}
\bibliography{OptimalTransport2}
\end{document}